\documentclass[11pt, twoside]{article}

\usepackage{graphicx,xcolor}
\usepackage[caption=false]{subfig}
\usepackage{amsmath}
\usepackage{amssymb,amsfonts}
\usepackage{amsthm}
\usepackage{bm}
\usepackage{mathrsfs}
\usepackage{amssymb}
\usepackage{multirow}

\newcommand{\Talpha}{{\bm T}}
\newcommand{\tildeTalpha}{{\bm T}_h}
\newcommand{\gzero}{{\bm w}}
\newcommand{\tildegzero}{{\bm w}_h}
\newcommand{\tildeE}{{\bm E}_h}
\newcommand{\tildeD}{{\bm D}_h}

\newcommand{\projLdh}{\bm Q_h}

\newcommand{\matriz}[1]{\mathcal{#1}}

\usepackage[a4paper,left=1in]{geometry}

\numberwithin{equation}{section}
\usepackage{multirow}
\usepackage{makecell}
\usepackage{booktabs}
\usepackage{cite}
\usepackage{hyperref}
\usepackage[nameinlink]{cleveref}

\theoremstyle{definition}
\newtheorem{theorem}{Theorem}[section]

\newtheorem{lemma}[theorem]{Lemma}

\newtheorem{definition}[theorem]{Definition}
\newtheorem{remark}[theorem]{Remark}
\newtheorem{example}[theorem]{Example}

\usepackage[ruled]{algorithm2e}

\newcommand{\vertiii}[1]{{\left\vert\kern-0.25ex\left\vert\kern-0.25ex\left\vert #1
		\right\vert\kern-0.25ex\right\vert\kern-0.25ex\right\vert}}

\usepackage{fancyhdr}
\begin{document}
	
\title{Analysis and approximations of Dirichlet boundary control of Stokes flows in the energy space
\thanks{
W. Gong was supported in part by the Strategic Priority Research Program of Chinese Academy of
				Sciences (Grant No. XDB 41000000), the National Key Basic Research Program (Grant No. 2018YFB0704304) and
				the National Natural Science Foundation of China (Grant No. 11671391 and 12071468). M. Mateos was supported by the Spanish Ministerio de Econom\'{\i}a y Competitividad under project MTM2017-83185-P.  Y.\ Zhang are partially supported by the US National Science Foundation (NSF) under grant number DMS-1818867.
}}

\author{Wei Gong\thanks{The State Key Laboratory of Scientific and Engineering Computing, Institute of Computational Mathematics \& National Center for Mathematics and Interdisciplinary Sciences, Academy of Mathematics and Systems Science, Chinese Academy of Sciences, 100190 Beijing, China. Email address: wgong@lsec.cc.ac.cn}	
	\and Mariano Mateos\thanks{Dpto. de Matem\'aticas. Universidad de Oviedo, Campus de Gij\'on, Spain. Email address: mmateos@uniovi.es}	
	\and John R.\ Singler\thanks{Department of Mathematics
		and Statistics, Missouri University of Science and Technology, Rolla, MO. Email address: singlerj@mst.edu}
	\and Yangwen Zhang\thanks{Department of Mathematical Science, University of Delaware, Newark, DE. Email address: ywzhangf@udel.edu}
}
	
	\date{\today}
	
	\maketitle

\begin{abstract}
 We study Dirichlet boundary control of Stokes flows in 2D polygonal domains. We consider cost functionals with two different boundary control regularization terms: the $L^2$ norm and an energy space seminorm. We prove well-posedness and regularity results for both problems, develop finite element discretizations for both problems, and prove finite element error estimates for the latter problem. The motivation to study the energy space problem follows from our analysis: we prove that the choice of the control space ${\bm L}^2(\Gamma)$ can lead to an optimal control with discontinuities at the corners, even when the domain is convex. We observe this phenomenon in numerical experiments. This behavior does not occur in Dirichlet boundary control problems for the Poisson equation on convex polygonal domains, and may not be desirable in real applications. For the energy space problem, we derive the first order optimality conditions, and show that the solution of the control problem is more regular than the solution of the problem with the ${\bm L}^2(\Gamma)$ regularization. We also prove a priori error estimates for the control in the energy norm, and present several numerical experiments for both control problems on convex and nonconvex domains.
\end{abstract}

\section{Introduction}

PDE-constrained optimal control is an active research area and has been popular for the last several decades.  Interest in analysis and computation for problems in this area has been generated by a wide variety of applications and the fast development of computational resources. There are already several monographs and chapters devoted to various aspects of the field, including theoretical analysis, computational methods, and application areas; see, e.g., \cite{HinzePinnauUlbrich2009,Lions,Casas-Mateos2017}.

Boundary control problems for PDEs are a very important part of this field since for many applications control may only be applied at the boundary of the physical domain.  Dirichlet boundary control problems are especially important in application areas, but the problems can be difficult to analyze mathematically -- especially when the physical domain has a nonsmooth boundary.  One of the key points in the study of Dirichlet boundary control problems is the choice of the control penalty in the cost functional.  A natural goal in many applications is to minimize the ``amount'' of control used, which naturally leads to a boundary control penalty using the $L^2(\Gamma)$ norm.  This also appears to be a reasonable choice from a numerical approximation point of view.  However, in the analysis of such a problem the governing state equation is typically understood in a very weak sense since the Dirichlet boundary condition is only in $L^2(\Gamma)$.

Despite this difficulty, many researchers have considered problems using the $L^2(\Gamma)$ control penalty and developed numerical methods and numerical analysis results.  One of the first contributions was the study of a finite element method for elliptic Dirichlet boundary control problems in \cite{FrenchKing1991}. Control constrained problems governed by semilinear elliptic equations on polygonal domains were studied in \cite{Casas_Raymond_First_SICON_2006}. Optimal-order error estimates were derived for the unconstrained problem in \cite{May_Rannacher_Vexler_High_SICON_2013} for both the control and state by introducing a dual control problem. Higher-order convergence rates were proved in \cite{Deckelnick_Andreas_Hinze_Three_SICON_2009} for {control-constrained} problems in smooth domains based on the superconvergence properties of regular triangulations. In \cite{Gong_Yan_Mixed_SICON_2011} the authors used a mixed finite element method for approximating the elliptic Dirichlet boundary control problem to avoid the very weak formulation of the state equation. For recent results on the regularity of solutions and standard finite element approximations of elliptic Dirichlet boundary control problems we refer to \cite{Apel_Mateos_Pfefferer_Regularity_SICON_2015}, \cite{Mateos2017} and the references cited therein. In \cite{APel_Mateos_Pfefferer_Arnd_DBC_MCRF_2018}, {optimal} error estimates on general (possibly nonconvex) polygonal domains {are obtained for quasi-uniform and superconvergence meshes.} Recently, the hybridizable discontinuous Galerkin (HDG) method has applied to the elliptic Dirichlet boundary control problem on convex domains \cite{ChenFuSinglerZhang,Chen_SINUM_20191,HuShenSinglerZhangZheng_HDG_Dirichlet_control1,HuMateosSinglerZhangZhang2,chen2020infty}. The HDG method also avoids the very weak formulation, and has a lower computational cost compared to traditional discontinuous Galerkin and mixed methods. We also refer to \cite{GongHinzeZhou2016,GongLi2017} for error estimates for parabolic Dirichlet boundary control problems, to \cite{Mateos-Neitzel2016} for state-constrained problems, and to \cite{Casas_Mateos_Raymond_Penalization_ESAIM_Control_Optim_2009} for a Robin penalization approach.

On the other hand, $H^{1/2}(\Gamma)$ appears to be a natural choice to study the state equation in the standard variational formulation. There are also some numerical analysis results in this direction.  The analysis of a finite element method for an elliptic Dirichlet boundary control problem in the energy space setting with $H^{1/2}(\Gamma)$ as the control space was performed in \cite{Of_Phan_Steinbach_Energy_NM_2015}; a boundary element method for this problem is proposed and analyzed in \cite{Of_Phan_Steinbach_Boundary_Element_MMAS_2010}. In \cite{Chowdhury_Thirupathi_Nandakumaran_Energy_2017} a variation to the energy space method is proposed where the control penalty now involves the harmonic extension of the control into the domain; a posteriori error estimates and the convergence of the adaptive finite element method is studied in \cite{Gong_Liu_Tan_Yan} for this approach. Also see \cite{Karkulik20} for another related approach to the energy space method. Sharp convergence rates for the energy space approach have recently been obtained in \cite{Winkler2020}. There are also other ways to deal with the inhomogeneous Dirichlet boundary condition. In \cite{Kunoth2002,Kunoth2005,Kunoth2009} elliptic Dirichlet boundary control problems are studied in the energy space setting using wavelet schemes for the spatial discretization and using a Lagrange multiplier for the inhomogeneous Dirichlet boundary condition.

Dirichlet boundary control problems are of great interest for applications in fluid dynamics; see, for example, \cite{Fursikov_Gunzburger_Hou_NS_SICON_1998,FGH,Gunzburger_Hou_Svobodny_NS_M2AN_1991,Gunzburger_Hou_Svobodny_Drag_SICON_1992,Hou_Ravindran_Penalized_SICON_1998,Hou_Ravindran_Penalty_SISC_1999,delosReyesKunisch2005,John_Wachsmuth_DBC_NFAO_2009,Ravindran_NS_M2AN_2017}. Although many numerical algorithms and simulation results can be found in the literature, there are very few well-posedness, regularity, and numerical analysis results for Dirichlet boundary control problems for fluid flows in polygonal domains.

In this work, we study Dirichlet Stokes flow control problems in 2D polygonal domains using both $ L^2 $ and $ H^{1/2} $ for the control spaces.  We give precise well-posedness and regularity results for both problems, and show that the $L^2 $ regularized optimal control can be discontinuous at the corners of a convex domain.  We prove higher regularity for the energy space control problem.  We also develop a finite element method for both problems, and prove a prior error estimates for the energy space problem.

Below, we give precise formulations of the Dirichlet Stokes control problems we consider and give a brief overview of related work.

Let $\Omega\subset\mathbb{R}^2$  be an open bounded domain with {polygonal} boundary $\Gamma$. We let $H^{m}(\Omega)$ denote the standard Sobolev space with norm $\|\cdot\|_{m,\Omega}$ and seminorm $|\cdot|_{m,\Omega}$, and we use bold font to denote vector valued spaces.  Set ${\bm H}^{m}(\Omega) = [H^{m}(\Omega)]^2$ and ${\bm  H}_0^1(\Omega) =\{{\bm v}\in {\bm  H}^1(\Omega);\ {\bm  v} = 0 \ \textup{on} \ \Gamma \}$. We denote the $L^2$-inner products on  ${\bm  L}^2(\Omega)$, $ L^2(\Omega)$, $ {\bm  L}^2(\Gamma)$ and $L^2(\Gamma)$ by
\begin{eqnarray*}
	({\bm  y},{\bm  z}) = \sum_{j=1}^2 \int_{\Omega}  y_j z_j,\qquad (p,q) =  \int_{\Omega} pq,\qquad ( {\bm  y},{\bm  z})_\Gamma = \sum_{j=1}^2\int_{\Gamma}  y_j  z_j,\qquad (u,v)_\Gamma = \int_\Gamma u v.
\end{eqnarray*}
We use $\langle\cdot,\cdot\rangle$ to denote the duality product between $H^{-s}(\Omega)$ and $H^s(\Omega)$.  We let $H^{s}(\Gamma)$ denote the space of traces of $H^{s+1/2}(\Omega)$ for $0<s<3/2$, and we note that $H^{s}(\Gamma)$ for $1/2 < s<3/2$ is given by $H^{s}(\Gamma) = \{ u \in \Pi_{i=1}^m H^{s}(\Gamma_i) : u \in C(\Gamma) \}$, see \cite[Theorem 1.5.2.8]{Grisvard1985}. (This definition does not make sense for $ s=3/2 $.)  For $ 0 < s < 3/2 $, we use $\langle\cdot,\cdot\rangle_\Gamma$ to denote the duality product between $H^{-s}(\Gamma)$ and $H^{s}(\Gamma)$.

For the Stokes problem, we use the standard spaces
\begin{eqnarray*}
	{\bm  H}(\text{div}; \Omega) = \{{\bm v} \in {\bm  L}^2(\Omega),\quad \nabla\cdot {\bm  v} \in  L^2(\Omega)\},  \quad  L_0^2(\Omega) = \left\{p \in L^2(\Omega),\quad  (p,1) = 0\right\},
\end{eqnarray*}
as well as the velocity spaces (see \cite[Section 2.1]{Raymond2007})
\begin{eqnarray*}
	{\bm  V}^s(\Omega) = \{{\bm y}\in {\bm  H}^s(\Omega):\ \nabla \cdot {\bm  y} = 0, \  \ \langle{\bm  y}\cdot{\bm  n},1\rangle_\Gamma=0\},  \quad  s\geqslant 0,
\end{eqnarray*}
which are Banach spaces with the ${\bm  H}^s(\Omega)$ norm.  For $0\leqslant s <3/2$, define
\begin{eqnarray*}
	{\bm  V}^s(\Gamma) =  \{{\bm u}\in {\bm  H}^s(\Gamma):\ ({\bm  u}\cdot{\bm  n},1)_\Gamma = 0\},
\end{eqnarray*}
and let ${\bm  V}^{-s}(\Gamma)$ denote the dual space.

For the control problem, consider a target state ${\bm  y}_d\in {\bm  H}$, a velocity penalty space ${\bm  H}\hookrightarrow {\bm  L}^2(\Omega)$, and a control penalty space ${\bm  U}\hookrightarrow {\bm  V}^0(\Gamma)$.  Let $ \alpha > 0 $ denote a Tikhonov regularization parameter, and consider the optimal control problem
\begin{equation}\label{P}
\min_{{\bm u}\in {\bm  U}} J({\bm  u})=\frac{1}{2}\|{\bm  y}_{\bm u}-{\bm  y}_d\|^2_{\bm  H} + \frac{\alpha}{2}\|\bm u\|_{\bm  U}^2,
\end{equation}
where ${\bm  y}_{\bm u}\in {\bm   V}^0(\Omega)$ is the unique solution  (either in the transposition sense, see  Definition \ref{D2.1} below, or standard variational solution) of the Stokes system
\begin{align}\label{StateEquationf}
\begin{split}
-\Delta{\bm  y}+\nabla p &={\bm  f} \quad  \text{in}\ \Omega,\\
\nabla\cdot{\bm  y}&=0 \quad  \text{in}\ \Omega,\\
{\bm  y}&={\bm  u} \quad  \text{on}\  \Gamma,\\
(p,1)&=0.
\end{split}
\end{align}

We note that similar Dirichlet control problems with various choices of the spaces $ \bm H $ and $ \bm U $ have been considered in the literature for both the Stokes and Navier-Stokes equations.  The choices ${\bm  H}={\bm  L}^4(\Omega)$ and ${\bm  U} = {\bm  V}^1(\Gamma)$ were used in the early work \cite{Gunzburger_Hou_Svobodny_NS_M2AN_1991}.  In \cite{delosReyesKunisch2005}, the spaces ${\bm  H} = {\bm  V}^1(\Omega)$ and ${\bm U} = {\bm  L}^2(\Gamma)$ are used for the objective functional; however, the optimal control problem looks for admissible optimal controls in ${\bm U}_\mathrm{ad} = {\bm   V}^{1/2}(\Gamma)$, which is the natural space for the controls to obtain a variational solution of the state equation \eqref{StateEquationf}.  In \cite{John_Wachsmuth_DBC_NFAO_2009}, the authors consider a smooth domain and choose ${\bm  H}= {\bm  V}^0(\Omega)$ and ${\bm  U}={\bm  V}^0(\Gamma)$. We show in polygonal domains that this approach leads to optimal controls that are \emph{discontinuous} at the corners; see \Cref{Stokes_L2} for the well-posedness and regularity results. However, a better regularity result for these spaces is obtained if we consider tangential control, i.e., we impose the condition $\bm u\cdot\bm n=0$ pointwise instead of $( \bm u\cdot \bm n,1)_\Gamma =0$, see \cite{GongHuMateosSinglerZhang2018} for more details.

Here we focus on the energy space method for the problem in polygonal domains. In \Cref{Stokes_enegy_space} we formulate the Dirichlet boundary control problem of Stokes equation with velocity space ${\bm  H}= {\bm  V}^0(\Omega)$ and control space ${\bm  U}=\bm V^{1/2}(\Gamma)$,  and we derive the first order optimality condition by using the Steklov-Poincar\'e operator. Higher regularity of the solutions is shown compared to the $\bm L^2(\Gamma)$ setting.  In \Cref{FEM_in_E} we give finite element approximations and error estimates for the energy space method. Numerical experiments are carried out in \Cref{Numer_ex} for both choices ${\bm  U}={\bm  V}^0(\Gamma)$ and ${\bm  U}=\bm V^{1/2}(\Gamma)$ in both convex and nonconvex polygonal domains.

\begin{remark}\label{Remark1}
	For $\bm f\in \bm H^{-1}(\Omega)$, if we let $\bm y^f\in \bm V^1(\Omega)\cap \bm H^1_0(\Omega)$ be the unique solution of \eqref{StateEquationf} for $\bm u = 0$ and redefine $\bm y_d:=\bm y_d-\bm y^f$, we can formulate an equivalent problem to \eqref{P} with $\bm f =0$, in the sense that the optimal control will be the same for both problems and the optimal states will differ by $\bm y^f$. Thus, in the rest of the work, we assume $\bm f =0$.
\end{remark}

\begin{remark}
	The introduction of control constraints does not lead to any differences in the regularity of the solutions or the rates of convergence. Control constrained problems can be treated by means of variational inequalities instead of equalities and there are plenty of examples about this in the literature. We focus on the unconstrained problem in order to avoid additional technicalities.
\end{remark}
\section{Regularity results}\label{S2}

We first summarize the result we presented in \cite{GongHuMateosSinglerZhang2018} about the concept of solution for Dirichlet data in ${\bm  V}^0(\Gamma)$ and its precise regularity.

Definitions of very weak solutions of the Stokes and Navier-Stokes equations for data in ${\bm  V}^0(\Gamma)$ and even ${\bm  V}^{-1/2}(\Gamma)$ have been given for convex polygonal domains and smooth domains; see \cite[Appendix A]{Conca1987_II}, \cite{Marusic-Paloca2000}, \cite[Appendix A]{Raymond2007}, and \cite[Definition 2.1]{John_Wachsmuth_DBC_NFAO_2009}.  In \cite{GongHuMateosSinglerZhang2018}, we showed how to extend the concept to problems posed on nonconvex polygonal domains for data in ${\bm  V}^{s}(\Gamma)$ with some negative $s$, and  we also proved that the optimal regularity ${\bm  V}^{s+1/2}(\Omega)$ expected for the solution can be achieved. In \cite{MoussaouiZine1998} a similar result is provided for convex polygonal domains, but only suboptimal regularity ${\bm  V}^{s+1/2-\varepsilon}(\Omega)$ for $ \varepsilon>0$ is proved.

To introduce the definition of solution of the state equation, we first need some results about the following \emph{compressible} Stokes equation:
\begin{equation}\label{StokesRegularity}
\begin{split}
-\Delta{\bm  z}+\nabla q &={\bm  g} \quad  \text{in}\ \Omega,\\
\nabla\cdot{\bm  z}&=h \quad  \text{in}\ \Omega,\\
{\bm  z}&=0 \quad  \text{on}\  \Gamma,\\
(q,1)&=0.
\end{split}
\end{equation}
For {data $({\bm g},h)\in \bm H^{-1}(\Omega)\times \bm L^2_0(\Omega)$,} this problem must be understood in the weak sense: Find $({\bm  z}_{{\bm g},h}, q_{{\bm g},h})\in {\bm H}^1_0(\Omega)\times L^2_0(\Omega)$ satisfying
\begin{align*}
	(\nabla\bm z_{{\bm g},h}, \nabla\bm \zeta) - (q_{{\bm g},h},\nabla\cdot \bm\zeta) &= (\bm g, \bm \zeta)\ \quad\forall {\bm \zeta}\in {\bm H}^1_0(\Omega), \\
	(\chi,\nabla\cdot\bm z_{{\bm g},h}) &= (h,\chi)\ \quad\forall {\chi}\in L^2_0(\Omega).
\end{align*}

Following \cite{Dauge1989}, we define the singular exponent $\xi$ as the real part of the smallest root different from zero  of the equation
	\begin{equation}\label{singular_ex}
	\sin^2(\lambda\omega)-\lambda^2\sin^2\omega=0,
	\end{equation}
	where $\omega$ denotes the greatest interior angle of $\Gamma$.
A numerical computation of $\xi$ shows, cf. \cite[Figure 2]{Dauge1989}, that $\xi\in(0.5,4]$, $\omega\mapsto\xi$ is strictly decreasing, $\xi >\pi/\omega$ if $\omega <\pi$, and $\xi <\pi/\omega$ if $\omega >\pi$.
Let
\begin{align}\label{def_s_star}
s^\star  =\min\{\xi-1/2,1/2\}.
\end{align}

\begin{theorem}\label{Dauge55a}\cite[Theorem 5.5 (a)]{Dauge1989}
Let $s$ satisfy $-1/2< s<s^\star$.  If $\bm g\in \bm H^{s-1/2}(\Omega)$ and $h\in H^{s+1/2}(\Omega)\cap L^2_0(\Omega)$, then \Cref{StokesRegularity} has a unique solution $({\bm  z}_{{\bm g},h}, q_{{\bm g},h}) \in [{\bm  H}^{3/2+s}(\Omega)\cap {\bm  H}_0^1(\Omega)]\times [H^{1/2 + s}(\Omega)\cap L_0^2(\Omega)]$. Moreover, we have
	\begin{align}\label{Continuity}
		\|{\bm  z}_{{\bm g},h}\|_{{\bm H}^{3/2+s}(\Omega)}+\|q_{{\bm g},h}\|_{H^{1/2+s}(\Omega)/\mathbb{R}} \leqslant C \big(\|{\bm  g}\|_{{\bm H}^{s-1/2}(\Omega)} + \|h\|_{H^{s+1/2}(\Omega)/\mathbb R}\big).
	\end{align}
\end{theorem}

Notice that although the pressure is uniquely determined as a function with the condition $(q,1)=0$, the norm must be taken modulo constant functions. It is important to note that \Cref{Dauge55a} only holds for  $s<1/2$. This means, even in convex  domains one cannot expect in general to have ${\bm  H}^2(\Omega)$ regularity of ${\bm  z}$.

	The ${\bm  H}^2(\Omega)$ regularity of $ \bm z $ can be obtained by requiring an additional condition on the divergence of ${\bm  z}$.  For example if $h\in H^1_0(\Omega)$ with $(h,1)=0$, then the above result holds for $ s = s^\star $ (this follows from \cite[Theorem 5.5(c)]{Dauge1989}, or the early reference \cite{Kellog-Osborn1976} for convex polygonal domains). This implies in a convex domain we have ${\bm  z}\in {\bm  H}^2(\Omega)$.
	
	This ${\bm  H}^2(\Omega)$ regularity result was used in \cite{Conca1987_II, MoussaouiZine1998} to define very weak solutions in polygonal domains using $h\in H^1_0(\Omega)$ as a test function.  Although this approach does enable us to define the transposition solution, it does not lead to optimal regularity results for the solution of the Dirichlet control problem.

Later, we also require a regularity result for the case $h\equiv 0$ in a convex domain. Let $(\bm z(\bm g),q(\bm g))$ denote the solution of \eqref{StokesRegularity} for $h=0$, i.e., $\bm z(\bm g) = \bm z_{\bm g,0}$ and $q(\bm g)=q_{\bm g,0}$.
\begin{theorem}\label{Dauge55b}\cite[Theorem 5.5(b)(c)]{Dauge1989}
	Suppose $\bm g \in \bm H^{t-1}(\Omega)$ for some $-1\leqslant t<\xi$ and $h=0$. {If $ \Omega $ is convex,} then the \emph{incompressible} Stokes equation
\begin{equation}
\begin{split}
-\Delta{\bm  z}+\nabla q &={\bm  g} \quad  \text{in}\ \Omega,\\
\nabla\cdot{\bm  z}&=0 \quad  \text{in}\ \Omega,\\
{\bm  z}&=0 \quad  \text{on}\  \Gamma,\\
(q,1)&=0
\end{split}
\end{equation}
has a unique solution
 $\bm z(\bm g)\in\bm V^{t+1}(\Omega)\cap \bm H^1_0(\Omega)$, $q(\bm g)\in H^t(\Omega)\cap L^2_0(\Omega)$, which satisfies
	\[
	\|\bm z(\bm g)\|_{\bm H^{1+t}(\Omega)}+\|q(\bm g)\|_{H^{t}(\Omega)} \leqslant C \|\bm g\|_{\bm H^{t-1}(\Omega)}.
	\]
\end{theorem}


Below, we derive the weak variational form for the state equation (\ref{StateEquationf}). Since the problem is linear, we may decompose the solution into the contributions from the right hand side ${\bm  f}$ and the Dirichlet boundary data ${\bm  u}$. The existence of a unique classical variational solution for ${\bm  f}\in {\bm  L}^2(\Omega)$ is standard and so we may set ${\bm  f}={\bm 0}$, see \Cref{Remark1}.

We use interpolation below to give precise regularity results for the state equation (with $ {\bm  f}={\bm 0} $), and therefore we define very weak solutions in the case ${\bm  u}\in {\bm  V}^{-s}(\Gamma)$ for $0<s<s^\star$.  Elements of this space do not necessarily satisfy any condition analogous to $({\bm  u}\cdot {\bm  n},1 )_\Gamma=0$.  In order to account for the constants, we follow \cite[Eq. (2.2)]{Raymond2007} and for $({\bm  z},q)\in {\bm  H}^{3/2+s}(\Omega)\times H^{1/2+s}(\Omega)$ with $s>0$ we define the constant
\begin{equation}
\label{Ray2.2}\lambda({\bm  z},q) = \frac{1}{|\Gamma|}( \partial_{\bm  n} {\bm  z}\cdot {\bm  n}-q,1)_\Gamma.
\end{equation}
This constant satisfies
\[\|\partial_{\bm  n} {\bm  z}-q {\bm  n}\|_{L^2(\Gamma)/\mathbb R} = \|\partial_{\bm  n} {\bm  z}-q {\bm  n}-\lambda({\bm  z},q){\bm  n}\|_{L^2(\Gamma)},\]
and we have
\[\partial_{\bm  n} {\bm  z}-(q+\lambda({\bm  z},q)) {\bm  n}\in \bm V^0(\Gamma).\]
This fact, trace theory, and \eqref{Continuity} give that for $0 < s<1/2$ we have
\begin{equation}\label{NormalTraceContinuity}
\|\partial_{\bm  n} {\bm  z}_{{\bm g}, h}-(q_{{\bm g}, h} +\lambda({\bm  z}_{{\bm  g}, h},q_{{\bm g}, h})) {\bm  n}\|_{H^s(\Gamma)}\leqslant C \big(\|{\bm  g}\|_{{\bm H}^{s-1/2}(\Omega)} + \|h\|_{H^{s+1/2}(\Omega)/\mathbb R}\big).
\end{equation}
This allows us to give the following well-defined notion of transposition solution for the state equation (again, with $ \bm{f} = \bm{0} $).

\begin{definition}\label{D2.1}Suppose $0\leqslant s<s^\star$ and ${\bm  u}\in {\bm  V}^{-s}(\Gamma)$. We say that ${\bm  y}_{\bm  u}\in {\bm  V}^0(\Omega)$, $p_{\bm u}\in \left(H^{1}(\Omega)\cap L^2_0(\Omega)\right)'$ is a solution in the transposition sense of
\begin{align}\label{StateEquation}
\begin{split}
-\Delta{\bm  y}+\nabla p &={\bm  0} \quad  \text{in}\ \Omega,\\
\nabla\cdot{\bm  y}&=0 \quad  \text{in}\ \Omega,\\
{\bm  y}&={\bm  u} \quad  \text{on}\  \Gamma,\\
(p,1)&=0,
\end{split}
\end{align}
if
	\begin{equation}\label{veryWeakForm}
	({\bm  y}_{\bm  u},{\bm  g}) - \langle {p}_{\bm u},{h}\rangle = \langle{\bm  u},{-\partial_{\bm  n}{\bm  z}_{{\bm g},h}+(q_{{\bm g},h}+\lambda({\bm  z}_{{\bm g},h},q_{{\bm g},h})){\bm  n}}\rangle_\Gamma,
	\end{equation}
	for all ${\bm  g}\in {\bm  L}^2(\Omega)$  and $h\in H^1(\Omega)\cap L^2_0(\Omega)$,
	where $({\bm  z}_{{\bm g},h},q_{{\bm g},h})\in {\bm  H}^1_0(\Omega) \times L^2_0(\Omega)$ is the unique  solution of \eqref{StokesRegularity} and $\lambda({\bm  z}_{{\bm g},h},q_{{\bm g},h})$ is the constant given in \eqref{Ray2.2}.
\end{definition}
This definition can be formally obtained by integrating by parts twice in the equation and also once in the divergence free condition.  We note that this definition can be written in different ways: using two separate equations tested by $ \bm g $ and $ h $ (see \cite{John_Wachsmuth_DBC_NFAO_2009} or \cite{Raymond2007}), or as one equation (see \cite{Conca1987_II} or \cite{Marusic-Paloca2000}).

Furthermore, this definition can be rewritten in different forms when $ \bm u $ is more regular.  First, if ${\bm  u}\in {\bm  V}^0(\Gamma)$, then $(\bm u,\lambda\bm n)_\Gamma = 0$ for every constant $\lambda\in\mathbb{R}$ and therefore \eqref{veryWeakForm} can be written as
\begin{align}\label{veryWeakForm0}
({\bm  y}_{\bm  u},{\bm  g}) - \langle{p}_{\bm u},{h}\rangle = ({\bm  u},{-\partial_{\bm  n}{\bm  z}_{{\bm g},h}+q_{{\bm g},h}{\bm  n}})_\Gamma.
\end{align}
Second, if ${\bm u}\in {\bm V}^{1/2}(\Gamma)$, then the very weak solution is the variational solution of the problem:  Find  $({\bm  y}_{{\bm u}}, p_{{\bm u}})\in {\bm H}^1(\Omega)\times L^2_0(\Omega)$ satisfying
\begin{align}\label{WeakForm}
	\begin{split}
	(\nabla\bm y_{{\bm u}}, \nabla\bm \zeta) - (p_{{\bm u}},\nabla\cdot \bm \zeta) &= 0\ \quad \forall {\bm \zeta}\in {\bm H}^1_0(\Omega), \\
	(\chi,\nabla\cdot\bm y_{{\bm u}}) &= 0\ \quad\forall {\chi}\in L^2(\Omega)/\mathbb R,\\
    \bm y_{\bm u} &= \bm u\quad \mbox{ on }\Gamma.
	\end{split}
\end{align}

Next, we give a regularity result for the state equation \eqref{StateEquation} on polygonal domains from \cite[{Theorem 2.2}]{GongHuMateosSinglerZhang2018}.  We note that an analogous result for smooth domains is found in \cite[Corollary A.1]{Raymond2007}. The limiting cases $s=-1/2$ and $s=3/2$ can be achieved when the domain is smooth; however, this is not possible for polygonal domains.

\begin{theorem}\label{continuous_mapping}
	If ${\bm u}\in {\bm V}^s(\Gamma)$ for $ -s^\star< s <{s^\star+1}$, then the solution of \eqref{StateEquation} satisfies
	\begin{align*}
		{\bm y}_{\bm u}\in {\bm V}^{s+1/2}(\Omega) \:\: \mbox{ and } \:\:
		p_{\bm u}\in
		\begin{cases}
			H^{s-1/2}(\Omega)/\mathbb R&\mbox{ if }s\geqslant 1/2,\\
		    \left(H^{1/2-s}(\Omega)/\mathbb R\right)'&\mbox{ if }s\leqslant 1/2.
		\end{cases}
	\end{align*}
	Moreover, the control-to-state mapping ${\bm u}\mapsto {\bm y}_{\bm u}$ is continuous from ${\bm V}^s(\Gamma)$ to ${\bm V}^{s+1/2}(\Omega)$.
\end{theorem}

We also recall here the concept of stress force on the boundary as used in \cite{GunzburgerHou1992}. Let $(\bm\psi,\phi)$ be the solution of the \emph{incompressible} Stokes system with source $\bm g\in \bm L^2(\Omega)$ and Dirichlet data $\bm u\in\bm V^{1/2}(\Gamma)$, i.e., $\bm\psi = \bm z(\bm g) + \bm y_{\bm u}$ and $\phi = q(\bm g) + p_{\bm u}$, where $(\bm z(\bm g),q(\bm g))$ is the solution of \eqref{StokesRegularity} with $h=0$, and $ ({\bm  y}_{{\bm u}}, p_{{\bm u}}) $ is the solution of \eqref{WeakForm}.

For $ \bm g $ and $ \bm u $ as above, we define the stress force on the boundary $\bm t(\bm g,\bm u)$ related to $(\bm\psi,\phi)$ to be the unique solution of the variational problem:
\begin{align}\label{def_S}
	\langle\bm t(\bm g,\bm u),\bm \zeta\rangle_\Gamma =
	(\nabla\bm\psi,\nabla\bm \zeta)-(\phi,\nabla\cdot\bm \zeta) - (\bm g,\bm\zeta)\ \quad\forall\bm\zeta\in\bm H^1(\Omega).
\end{align}
Notice that for $\bm u\in \bm V^{r+1/2}(\Gamma)$ with $r>0$, integration by parts shows that
\begin{align}\label{def_S_1}
\bm t(\bm g,\bm u) = \partial_{\bm n}\bm \psi- \phi\bm n.
\end{align}

For $ 0\leqslant s < {s^\star+1}$, we define $\bm E:\bm V^s(\Gamma)\to\bm L^2(\Omega)$ by
\begin{align}\label{def_E}
\bm E \bm u=\bm y_{\bm u}.
\end{align}
 Directly from \eqref{veryWeakForm0} with $h=0$ and \eqref{def_S_1}, the adjoint $\bm E^\star:\bm L^2(\Omega)\to \bm V^{-s}(\Gamma)$ is defined by
 \begin{align}\label{def_Estar}
 \bm E^\star\bm g = -\partial_{\bm n}\bm z(\bm g) + q(\bm g)\bm n = -\bm t(\bm g,\bm 0).
 \end{align}
 By \Cref{continuous_mapping} we know that $\bm E: \bm V^s(\Gamma)\to\bm L^2(\Omega)$ is bounded and hence $\bm E^\star: \bm L^2(\Omega)\to \bm V^{-s}(\Gamma)$ is also bounded. Therefore, $\bm E^\star\bm E: \bm V^s(\Gamma)\to \bm V^{-s}(\Gamma)$ is bounded. Specifically, setting $s=1/2$ gives that for all $\bm u\in \bm H^{1/2}(\Gamma)$ we have
 \begin{align}\label{boundness_of_EstarE}
 \|\bm E^\star\bm E \bm u\|_{\bm H^{-1/2}(\Gamma)} \leqslant C\|\bm u\|_{\bm H^{1/2}(\Gamma)}.
 \end{align}


\section{Stokes Dirichlet boundary control in ${\bf V}^0(\Gamma)$}\label{Stokes_L2}

In this section, we investigate the case ${\bm  U} = {\bm  V}^0(\Gamma)$.
For {$\bm y_d\in \bm L^2(\Omega)$} and $\alpha>0$, our control problem reads
\begin{equation}\label{P0}
\min_{\bm u\in\bm V^0(\Gamma)}\ J_{0}({\bm u})=\frac 1 2\|{\bm y}_{\bm u}-{\bm y}_d\|^2_{{\bm L}^2(\Omega)}+\frac{\alpha}{2}\|{\bm u}\|^2_{{\bm L}^2(\Gamma)},
\end{equation}
where ${\bm y}_{\bm u}\in {\bm V}^0(\Omega)$ is the solution of the state equation \eqref{veryWeakForm0}. By \eqref{def_E} we have
\begin{align}\label{J0}
	\begin{split}
		J_{0}({\bm u})&=\frac{1}{2}(\bm E^\star\bm E \bm u,\bm u)_\Gamma-(\bm E^\star\bm y_d,\bm u)_\Gamma+ \frac{c_\Omega}2+\frac{\alpha}{2}\|{\bm u}\|^2_{{\bm L}^2(\Gamma)}\\
		&:=	\bm F(\bm u)+ \frac{c_\Omega}2+\frac{\alpha}{2}\|{\bm u}\|^2_{{\bm L}^2(\Gamma)},
	\end{split}
\end{align}
where $c_\Omega = \|\bm y_d\|^2_{\bm L^2(\Omega)}$ and $\bm F(\bm u)=\frac{1}{2}(\bm E^\star\bm E \bm u,\bm u)_\Gamma-(\bm E^\star\bm y_d,\bm u)_\Gamma+ \displaystyle\frac{c_\Omega}{2}$ is the tracking term. It is straightforward to prove that
\begin{align}\label{derivative_F}
\bm F'(\bm u)\bm v = (\bm E^\star\bm E \bm u,\bm v)_\Gamma-(\bm E^\star\bm y_d,\bm v)_\Gamma\ \quad\forall{\bm u\in \bm V^0(\Gamma)\mbox{ and }\bm v\in \bm V^0(\Gamma)}.
\end{align}

Although we are mainly interested in this work in regularization in the energy space ${\bm V}^{1/2}(\Gamma)$, the solution properties of the problem with ${\bm V}^{0}(\Gamma)$-regularization are also of interest in order to more clearly see the advantages and disadvantages of energy space control problem. It is also interesting to see the differences between the Dirichlet boundary control of the Poisson equation (cf. \cite{Apel_Mateos_Pfefferer_Regularity_SICON_2015}) and of the Stokes system.

Using the strict convexity of the functional and the continuity of the control-to-state mapping, which follows from  \Cref{continuous_mapping}, it is standard to prove the existence of a unique solution ${\bm u}_0\in {\bm V}^0(\Gamma)$ of problem ($P_0$).  {We also prove regularity results below, and show that the optimal control can be discontinuous at the corners of a convex polygonal domain.}

{
\begin{theorem}\label{T2.8}Suppose  $\bm y_d\in \bm H^m(\Omega)$ for some $0\leqslant m < s^\star$ and let ${\bm u}_0\in {\bm V}^0(\Gamma)$ be the solution of problem ($P_0$). Then ${\bm u}_0\in {\bm V}^s(\Gamma)$ for all $0\leqslant s<s^\star$ and there exist ${\bm y}_0\in {\bm V}^{s+1/2}(\Omega)$, $p_0\in (H^{1/2-s}(\Omega)\cap L^2_0(\Omega))'$, $\bm z_0\in {\bm V}^{1+t}(\Omega)\cap \bm H^1_0(\Omega)$ and $q_0\in H^t(\Omega)\cap L^2_0(\Omega)$ for all $t\leqslant 1+m$ such that $t<\xi$, that satisfy the state equation
	
	\begin{equation}
	\begin{split}
	-\Delta\bm y_0+\nabla p_0 &=\bm 0 \quad  \text{in}\ \Omega,\\
	\nabla\cdot\bm y_0&=0 \quad  \text{in}\ \Omega,\\
	\bm y_0&=\bm u_0 \quad  \text{on}\  \Gamma,\\
	(p_0,1)&=0,
	\end{split}
	\end{equation}
	the adjoint state equation
	\begin{equation}
	\begin{split}
	-\Delta\bm z_0+\nabla q_0 &=\bm y_0 -\bm y_d \quad  \text{in}\ \Omega,\\
	\nabla\cdot\bm z_0&=0 \quad  \text{in}\ \Omega,\\
	\bm z_0&=0 \quad  \text{on}\  \Gamma,\\
	(q_0,1)&=0,
	\end{split}
	\end{equation}
	and the optimality condition
	\begin{eqnarray}
	(\alpha  {\bm u}_0-(\partial_{\bm n}\bm z_0-q_0{\bm n}),{\bm v})_\Gamma = 0\quad \forall {\bm v}\in {\bm V}^0(\Gamma).\label{opt}
	\end{eqnarray}
	Moreover, there exists $\lambda_0\in\mathbb{R}$ such that
\[\bm u_0 = \frac{1}{\alpha}(\partial_{\bm n}\bm z_0-(q_0+\lambda_0){\bm n}),\]
and
	\[{\bm u}_0\in \prod_{i=1}^{n} {\bm H}^{t-1/2}(\Gamma_i)\quad\mbox{ for all }t\leqslant m+1\mbox{ such that }t<\xi.\]
Finally, if $m > 0$ and $\Omega$ is convex, then $\bm u_0$ is continuous at a corner $x_j$ if and only if $q_0(x_j)+\lambda_0=0$.
\end{theorem}
Here, the state equation must be understood in the very weak sense \eqref{veryWeakForm0}, while the adjoint state equation must be understood in the variational sense.
\begin{proof}
	By the definition of $J_{0}(\bm u)$ in \eqref{J0} and \eqref{derivative_F}, the derivative of the objective functional $J_{0}(\bm u)$ for $\bm u, \bm v\in\bm V^0(\Gamma)$ can be written as
	\begin{align*}
	J_{0}'(\bm u)\bm v &= (\alpha \bm u + \bm E^\star\bm E \bm u,\bm v)_\Gamma-(\bm E^\star\bm y_d,\bm v)_\Gamma\\
	&= (\alpha \bm u + \bm E^\star(\bm E \bm u-\bm y_d),\bm v)_\Gamma\\
	& = (\alpha  {\bm u}-(\partial_{\bm n}\bm z(\bm y_{\bm u}-\bm y_d)-q(\bm y_{\bm u}-\bm y_d){\bm n}),{\bm v})_\Gamma,
	\end{align*}
	where we used \eqref{def_E} and \eqref{def_Estar} in the last equality.  The optimality conditions follow in a standard way.
	For $\bm v\in \bm V^0(\Gamma)$ we have that $(\lambda\bm u,\bm v)=0$ for any $\lambda\in\mathbb{R}$. Taking {$\lambda_0 $ to equal the constant $ \lambda(\bm z(\bm y_{\bm u}-\bm y_d),q(\bm y_{\bm u}-\bm y_d))$, which is defined in \eqref{Ray2.2},} we also have that
	\[(\alpha  {\bm u}_0-(\partial_{\bm n}\bm z_0-(q_0+\lambda_0){\bm n}),{\bm v})_\Gamma = 0\quad \forall {\bm v}\in {\bm V}^0(\Gamma).\]
	
	This implies that $\alpha \bm u_0$ is the {$\bm L^2(\Gamma)$-projection} of $\partial_{\bm n}\bm z_0-(q_0+\lambda_0){\bm n}$ onto $\bm V^0(\Gamma)$. Since $\partial_{\bm n}\bm z_0-(q_0+\lambda_0){\bm n}\in\bm V^0(\Gamma)$, we have \[\bm u_0 = \frac{1}{\alpha}(\partial_{\bm n}\bm z_0-(q_0+\lambda_0){\bm n}).\]
		
	The regularity follows from a bootstrapping argument:	
	From  \Cref{continuous_mapping} we have that ${\bm y}_0\in {\bm V}^{1/2}(\Omega)$. Using this and taking into account that ${\bm y}_d\in \bm H^m(\Omega)$, we have from \Cref{Dauge55b} that
	$\bm z_0\in {\bm V}^{1+t}(\Omega)$, $q_0\in H^{t}(\Omega)\cap L_0^2(\Omega)$ for all $t\leqslant 1+m$ such that $t<\xi$.
	
	From trace theory, and since $1/2<t$, it is clear that
	\[ \partial_{\bm n}\bm z_0-(q_0+\lambda_0)\bm n\in \prod_{i=1}^n {\bm H}^{t-1/2}(\Gamma_i)\quad \mbox{ for all } t\leqslant m+ 1\mbox{ such that }t<\xi.\]
	For $t<1$, and taking $s=t-1/2$, we have that $s<s^\star$ and that
	$\prod_{i=1}^n {\bm H}^{s}(\Gamma_i) = {\bm H}^{s}(\Gamma)$. Therefore, \eqref{opt} gives that $u_0\in {\bm H}^{s}(\Gamma)$ for all $s < s^\star$.
	The regularity of the optimal state follows from  \Cref{continuous_mapping}.
	
	If $m > 0$ and $\Omega$ is convex, then
	 the gradient of the dual pressure $q_0$ is a function in $H^{t-1}(\Omega)$ with $t-1>0$. So we have that each component $z^i$, $i=1,2$ of $\bm z_0$, satisfies $\-\Delta z^i\in H^{t-1}(\Omega)$ and $z^i=0$ on $\Gamma$. {Therefore, we have that $\partial_{\bm n}z^i(x_j)=0$, $i=1,2$, for every convex corner $x_j$ (cf. \cite[Appendix A]{Casas_Mateos_Raymond_Penalization_ESAIM_Control_Optim_2009}); also, from \cite[Lemma A2]{Casas_Mateos_Raymond_Penalization_ESAIM_Control_Optim_2009} and the Sobolev imbedding theorem we have that the normal derivative of $\bm z_0$ is a continuous function.}
	For the pressure, the situation is slightly different. From trace theory we have that $q_0\in H^{t-1/2}(\Gamma),$ and by Sobolev imbeddings we know $q_0$ is a continuous function.
	Nevertheless, the vector ${\bm n}$ is discontinuous at the corners, and hence the $(q_0+\lambda_0)\bm n$ can only be continuous at $x_j$ if $q_0(x_j)=-\lambda_0$.
\end{proof}
\begin{remark}\label{Remark2}
This regularity of the optimal control in a convex domain is essentially different from the regularity achieved by the optimal control of problems related to the Poisson equation.  The solution of a problem governed by the Poisson equation must be a continuous function, which is also zero at the corners. In our case, the optimal control may show discontinuities. See \Cref{Figure_02} for an example with a continuous control and  \Cref{Figure_03} for a problem example with \emph{discontinuous} control.
\end{remark}
\begin{remark}\label{Remark3}
Notice that the pressure is determined up to a constant. We choose the pressure such that $(q_0,1)=0$, but any other representative is of course possible. The value of $\lambda_0$ would change accordingly, so that $q_0+\lambda_0$ does not vary.
\end{remark}
}

\section{Stokes Dirichlet boundary control in the energy space}\label{Stokes_enegy_space}

Next, we consider Stokes Dirichlet boundary control with a different regularization term:
\begin{align}\label{OCP_Stokes_obj}
\min\limits_{{\bm u}\in {\bm V}^{1/2}(\Gamma)} J_{1/2}({\bm u})=\frac 1 2 \|{\bm y}_{\bm u}-{\bm y}_d\|^2_{{\bm L}^2(\Omega)}+
\frac{\alpha}{2}|{\bm u}|^2_{{\bm H}^{1/2}(\Gamma)},
\end{align}
{where again we assume $\bm y_d\in \bm L^2(\Omega)$ and $\alpha>0$.}

There are different kinds of definitions for the $\bm H^{1/2}(\Gamma)$-norm, e.g., one may use the Sobolev-Slobodeckii norm or the Fourier transform. The key point to the study of the optimization problem (\ref{OCP_Stokes_obj}) is to find an appropriate representation for the $\bm H^{1/2}(\Gamma)$-norm that enables us to derive the first order optimality condition. Here we follow the idea of \cite{Of_Phan_Steinbach_Energy_NM_2015} and {introduce a Stokes version of the Steklov-Poincar\'e operator (cf. \cite{BadeaDiscacciatiQuarteroni,DiscacciatiQuarteroni2003}) associated with (\ref{veryWeakForm0}).}

It follows from  \Cref{continuous_mapping} that  for any given control $\bm u\in \bm V^{1/2}(\Gamma)$,  there exists a unique state $(\bm y_{\bm u},p_{\bm u})\in \bm V^1(\Omega)\times L_0^2(\Omega)$ that satisfies
\begin{align}\label{bound_yu_pu}
\|{\bm y}_{\bm u}\|_{\bm H^1(\Omega)}+\|p_{\bm u}\|_{L^2(\Omega)}\leqslant C\|{\bm u}\|_{\bm H^{1/2}(\Gamma)}.
\end{align}
%
%
Given ${\bm u}\in {\bm V}^{1/2}(\Gamma)$, we define
 $\bm D\bm u\in \bm H^{-1/2}(\Gamma)$ by
	\begin{align}\label{Steklov-Poincare}
	\langle{\bm D}{\bm u},\bm v \rangle_\Gamma=(\nabla \bm y_{\bm u},\nabla {\bm R}\bm v)-(p_{\bm u},\nabla\cdot {\bm R}\bm v)\ \quad\forall \bm v\in \bm H^{1/2}(\Gamma),
	\end{align}
	where ${\bm R}$ is any continuous extension operator from $\bm H^{1/2}(\Gamma)$ to $\bm H^1(\Omega)$.
\begin{lemma}
	The definition of $\bm D$ is independent of the chosen extension $\bm R$ and
	\begin{subequations}
	\begin{align}
	\bm D\bm u &= \bm t (\bm 0,\bm u),\label{charac_of_D}\\
	\|\bm  D \bm u\|_{\bm H^{-1/2}(\Gamma)} &\leqslant C \|\bm u\|_{\bm H^{1/2}(\Gamma)}\ \quad\forall \bm u\in \bm V^{1/2}(\Gamma). \label{boundness_of_D}
	\end{align}		
	\end{subequations}
\end{lemma}
\begin{proof}
First of all, writing the partial differential equation in divergence form as
	\[-\nabla\cdot\big((\nabla+\nabla^T){\bm y}_{\bm u} -p_{\bm u} \mathcal{I}\big) = 0\]
	gives $(\nabla+\nabla^T){\bm y}_{\bm u} -p_{\bm u} \mathcal{I}\in {\bm H}({\rm div}; \Omega)$, and so this function has a well defined {normal} trace in $\bm H^{-1/2}(\Gamma)$. It is remarkable too that it is possible to define a variational normal derivative $\partial_{\bm n} {\bm y}_{\bm u}\in {\bm H}^{-1/2}(\Gamma)$, cf. \cite[Lemma A6]{Casas_Mateos_Raymond_Penalization_ESAIM_Control_Optim_2009}, and hence  $p_{\bm u}{\bm n}$ is also a well defined element in ${\bm H}^{-1/2}(\Gamma)$.
	
	Next, for all $\bm u, \bm v\in\bm H^{1/2}(\Gamma)$, integrating by parts in the definition of $ \bm{ D u } $ gives
	\begin{align}\label{char_Du}
	\langle{\bm D}{\bm u},\bm v \rangle_\Gamma&=  \int_\Omega
	\big(\nabla \bm y_{\bm u}\nabla {\bm R}\bm v- p_{\bm u}\nabla\cdot {\bm R}\bm v\big) \nonumber\\
	&=   \int_\Omega\big(-\Delta {\bm y}_{\bm u}+\nabla p_{ {\bm u}}\big)\bm R\bm v +\langle \partial_{\bm n} {\bm y}_{\bm u}-p_{\bm u}{\bm n},\bm v\rangle_\Gamma\nonumber\\
	&=  \langle \partial_{\bm n} {\bm y}_{\bm u}-p_{\bm u}{\bm n},\bm v\rangle_\Gamma,
	\end{align}
	where we used $-\Delta {\bm y}_{\bm u}+\nabla p_{\bm u} = 0$. This proves that the definition of $\bm D$ is independent of the chosen extension $\bm R$, and \eqref{charac_of_D} holds by \eqref{def_S_1} and \eqref{char_Du}.
	
	Finally, we prove \eqref{boundness_of_D}. Using the definition of $\bm D$ in \eqref{Steklov-Poincare}, the bound in \eqref{bound_yu_pu}, and the continuity of $\bm R: \bm H^{1/2}(\Gamma)\to \bm H^{1}(\Omega)$ gives
	\begin{align*}
		\|\bm  D \bm u\|_{\bm H^{-1/2}(\Gamma)} &= \sup_{\bm 0\neq \bm v\in\bm H^{1/2}(\Gamma)}\frac{	\langle{\bm D}{\bm u},\bm v \rangle_\Gamma}{\|\bm  v\|_{\bm H^{1/2}(\Gamma)}} \\
		&\leqslant C \sup_{\bm 0\neq \bm v\in\bm H^{1/2}(\Gamma)}\frac{	(\|{\bm y}_{\bm u}\|_{\bm H^1(\Omega)}+\|p_{\bm u}\|_{L^2(\Omega)}) |{\bm R}\bm v|_{\bm H^1(\Omega)}}{\|\bm  v\|_{\bm H^{1/2}(\Gamma)}}\\
		&\leqslant C(\|{\bm y}_{\bm u}\|_{\bm H^1(\Omega)}+\|p_{\bm u}\|_{L^2(\Omega)})\\
		&\leqslant C \|\bm u\|_{\bm H^{1/2}(\Gamma)}.
	\end{align*}
\end{proof}

\begin{lemma} $\langle{\bm D}{\bm u},\bm u \rangle_\Gamma^{{1/2}}$ is a seminorm in ${\bm V}^{1/2}(\Gamma)$ equivalent to the ${\bm H}^{1/2}(\Gamma)$ seminorm.
\end{lemma}
\begin{proof}
Let $\bm Q$ be the projection of $\bm H^{1/2}(\Gamma)$ onto $\bm V^{1/2}(\Gamma)$ and set  $\bm R\bm v=\bm y_{\bm Q\bm v}$. Notice that $\nabla\cdot \bm R{\bm v} =0$  and if $\bm v\in\bm V^{1/2}(\Gamma)$ then $\bm R\bm v=\bm y_{\bm v}$.
      By \eqref{Steklov-Poincare} we have
\begin{align}\label{EM2}
\langle{\bm D}{\bm u},\bm v \rangle_\Gamma=(\nabla \bm y_{\bm u},\nabla {\bm y}_{\bm v})\ \quad\forall \bm v\in \bm V^{1/2}(\Gamma),
\end{align}
and thus we have that $\langle{\bm D}{\bm u},\bm u \rangle_\Gamma^{{1/2}}$ is a seminorm in ${\bm V}^{1/2}(\Gamma)$  equivalent to the ${\bm H}^{1/2}(\Gamma)$ seminorm.
\end{proof}

Proceeding similarly to the derivation of \eqref{J0}, the precise formulation of our control problem is given by
\begin{align}\label{P12}
	\begin{split}
		\min\limits_{{\bm u}\in {\bm V}^{1/2}(\Gamma)} J_{1/2}({\bm u}) &=\frac{1}{2}\|{\bm y}_{\bm u}-{\bm y}_d\|^2_{{\bm L}^2(\Omega)}+
		\frac{\alpha}{2}\langle \bm D{\bm u},{\bm u}\rangle_\Gamma\\
		& = \frac 1 2\langle {\Talpha} {\bm u},{\bm u}\rangle_\Gamma-\langle {\gzero},{\bm u}\rangle_\Gamma + \frac{c_\Omega}{2},
	\end{split}
\end{align}
where $c_\Omega = \|{\bm y}_d\|_{\bm L^2(\Omega)}^2$ and
\begin{align}\label{def_T}
{\Talpha} = \alpha {\bm D}+{\bm E}^\star{\bm E}, \quad {\gzero}={\bm E}^\star{\bm y}_d\in {\bm V}^{-{1/ 2}}(\Gamma).
\end{align}
The functional being convex and coercive implies that problem \eqref{P12} has a unique solution $\bar {\bm u}\in \bm V^{1/2}(\Gamma)$.

We also note that, by \eqref{EM2}, an alternative way to write the functional for ${\bm u}\in {\bm V}^{1/2}(\Gamma)$ is
\begin{align*}
J_{1/2}(\bm u) = \frac{1}{2}\|{\bm y}_{\bm u}-{\bm y}_d\|^2_{{\bm L}^2(\Omega)}+
\frac{\alpha}{2}\| \nabla {\bm y}_{\bm u}\|^2_{\bm L^2(\Omega)}.
\end{align*}

\begin{lemma}\label{property_T}
There exist constants $C_1,C_2>0$ such that for every $\bm u,\bm v\in \bm V^{1/2}(\Gamma)$
\[\langle {\Talpha} {\bm u},{\bm v}\rangle_\Gamma\leq C_1\|\bm u\|_{\bm V^{1/2}(\Gamma)}\|\bm v\|_{\bm V^{1/2}(\Gamma)}\]
and
\[\langle {\Talpha} {\bm u},{\bm u}\rangle_\Gamma\geq C_2\|\bm u\|^2_{\bm V^{1/2}(\Gamma)}.\]

\end{lemma}
\begin{proof}
The first property follows immediately from the definition of $\Talpha$. Notice that $\bm D$ maps $\bm V^{1/2}(\Gamma)$ into $\bm H^{-1/2}(\Gamma)$ which is continuously embedded in $\bm V^{-1/2}(\Gamma)$ by duality.

Next, by  \eqref{def_T}, \eqref{def_E} and  \eqref{EM2} we have
	\begin{align*}
		\langle {\Talpha} {\bm u},{\bm u}\rangle_\Gamma&=\|{\bm E}{\bm u}\|_{\bm L^2(\Omega)}^2+\alpha\langle {\bm D} {\bm u},{\bm u}\rangle_\Gamma \\
		&=\|{\bm y}_{\bm u}\|_{\bm L^2(\Omega)}^2+\alpha \|\nabla {\bm y}_{\bm u}\|_{\bm L^2(\Omega)}^2\\
		&\geqslant \min(1,\alpha)\|{\bm y}_{\bm u}\|_{{\bm H}^1(\Omega)}^2\\
		&\geqslant C_2\|{\bm u}\|_{{\bm H}^{1/2}(\Gamma)}^2 = C_2\|{\bm u}\|_{{\bm V}^{1/2}(\Gamma)}^2,
	\end{align*}
	where we  used the trace theorem in the last inequality.
\end{proof}

Next,  we  give more insights into the structure of the solution to problem ($P_{1/2}$).
The functional $J_{1/2}$ in problem (\ref{P12}) is Frech\'et differentiable with respect to $\bm u$. Furthermore, for all $\bm u, {\bm v}\in {\bm V}^{1/2}(\Gamma)$,  by \eqref{P12} and \eqref{def_T} we have
\begin{align*}
J_{1/2}'({\bm u}){\bm v} &=\langle {\Talpha} {\bm u} -{\gzero},{\bm v}\rangle_\Gamma\\
& =\langle \alpha \bm D\bm u + \bm E^\star(\bm E \bm u-{\bm y}_d), \bm v \rangle_{\Gamma}\\
& = \langle \alpha(\partial_{\bm n} {\bm y}_{\bm u}-p_{\bm u}{\bm n})-(\partial_{\bm n} \bm z(\bm y_{\bm u}-\bm y_d)-q(\bm y_{\bm u}-\bm y_d){\bm n}),{\bm v}\rangle_\Gamma,
\end{align*}
where we used \eqref{char_Du}, \eqref{def_E} and  \eqref{def_Estar} in the last equality.

Now we are in the position to derive the regularity of the solution to the minimization problem \eqref{P12}.
\begin{theorem}\label{Theorem4.1}
	Assume ${\bm y}_d\in {\bm H}^{m}(\Omega)$ for some $0\leqslant m <\min\{2,1+\xi\}$, and
	let $\bar{\bm u}\in \bm V^{1/2}(\Gamma)$ be the optimal solution of problem (\ref{P12}). Then $\bar{\bm u}\in {\bm V}^{{1/ 2}+{r}}(\Gamma)$ for all $r<\min\{1,\xi\}$ and there exist $\bar {\bm y}\in {\bm V}^{1+r}(\Omega)$, $\bar p\in H^{r}(\Omega)\cap L_0^2(\Omega)$, $\bar {\bm z}\in {\bm V}^{1+t}(\Omega)\cap \bm H^1_0(\Omega)$ and $\bar q\in H^t(\Omega)\cap L_0^2(\Omega)$ for all $t\leqslant 1+m$ such that $t<\xi$ that satisfy the state equation
	\begin{align*}
	-\Delta\bar{\bm y}+\nabla \bar p &=\bm 0 \quad  \text{in}\ \Omega,\\
	\nabla\cdot\bar{\bm y}&=0 \quad  \text{in}\ \Omega,\\
	\bar{\bm y}&=\bar{\bm u} \quad  \text{on}\  \Gamma,\\
	(\bar p,1)&=0,
	\end{align*}
	the adjoint state equation
	\begin{align*}
	-\Delta\bar{\bm z}+\nabla \bar q &=\bar{\bm y} -\bm y_d \quad  \text{in}\ \Omega,\\
	\nabla\cdot\bar{\bm z}&=0 \quad  \text{in}\ \Omega,\\
	\bar{\bm z}&=0 \quad  \text{on}\  \Gamma,\\
	(\bar q,1)&=0,
	\end{align*}
	and the optimality condition
\begin{equation*}\langle\alpha(\partial_{\bm n} {\bar{\bm y}}-\bar p{\bm n}) - ( \partial_{\bm n} {\bar{\bm z}}-\bar q{\bm n}), \bm v\rangle_\Gamma = 0\ \quad\forall \bm v\in \bm V^{1/2}(\Gamma).\end{equation*}
Moreover, there exists $\bar\lambda\in\mathbb{R}$ such that
\begin{equation}\label{OCP_Stokes_u_Neumann}\alpha(\partial_{\bm n} {\bar{\bm y}}-\bar p{\bm n}) = \partial_{\bm n} {\bar{\bm z}}-(\bar q+\bar \lambda){\bm n}.
\end{equation}
\end{theorem}
Here, both the state equation and the adjoint state equation must be understood in the variational sense.
\begin{proof}
The minimization problem, being a convex problem, is equivalent to the following Euler-Lagrange equation
\begin{align}\label{Euler_Lagrange}
J_{1/2}'({\bm u}){\bm v}=\langle {\Talpha} {\bm u}-{\gzero},{\bm v}\rangle_\Gamma =0\quad\forall {\bm v}\in {\bm V}^{1/2}(\Gamma).
\end{align}
The existence of a unique solution follows immediately from the Lax-Milgram theorem and \Cref{property_T}.
First order optimality conditions follow in a standard way. Taking $\bar\lambda = \lambda(\bar z,\bar q)$, we deduce relation \eqref{OCP_Stokes_u_Neumann} as we did for the $\bm L^2(\Gamma)$-regularized problem.

Since $\bar {\bm u}\in {\bm V}^{1/2}(\Gamma)$, by  \Cref{continuous_mapping} we have that $\bar {\bm y}\in \bm V^1(\Omega)$. From  \Cref{Dauge55a,Dauge55b}, we obtain $\bar {\bm z}\in {\bm V}^{1+t}(\Omega)$ and $\bar q\in H^t(\Omega)\cap L_0^2(\Omega)$ for all $t\leqslant \min\{2,1+m\}$ with $t<\xi$. Using the trace theorem (see \cite[Theorem 1.5.2.1]{Grisvard1985}) we arrive at
	\begin{eqnarray}
	{\bm e}:=\partial_{\bm n} \bar {\bm z}-(\bar q+\bar\lambda){\bm n}\in \prod_{i=1}^{n} {\bm H}^{t-1/2}(\Gamma_i)\subset
\prod_{i=1}^{n} {\bm H}^{r-1/2}(\Gamma_i)\ \quad\forall r<\min\{1,\xi\}.\nonumber
	\end{eqnarray}

		From the trace theorem again on polygons, see \cite[Theorem 1.5.2.1]{Grisvard1985} and also \cite[Remark 1.1, Chapter 1]{GR}, we know that there exists some ${\bm Y}\in {\bm H}^{1+r}(\Omega)$ such that $\partial_{\bm n} {\bm Y}={\bm e}/\alpha$ on $\Gamma$. So we have that ${\bm F}=\Delta {\bm Y}\in {\bm H}^{r-1}(\Omega)$ and $H=-\nabla\cdot {\bm Y}\in H^r(\Omega)$. Using the state equation and the optimality condition (\ref{OCP_Stokes_u_Neumann}), we deduce that the pair $(\bar{\bm y}-{\bm Y},\bar p)$ satisfies
		\begin{align*}
		-\Delta (\bar {\bm y}-{\bm Y})+ \nabla\bar p={\bm F}\ {\rm in}\ \Omega,\quad \nabla\cdot (\bar{\bm y}-{\bm Y})=H\ {\rm in}\ \Omega,\quad \partial_{\bm n}(\bar{\bm y}-{\bm Y})-\bar p{\bm n}=0\ {\rm on}\ \Gamma.
		\end{align*}
		This problem has a variational solution, which is unique up to a constant. Noticing that the singular exponents for the Stokes problem with Neumann boundary conditions are the same as those for Dirichlet boundary conditions, see e.g. \cite[pp.\ 191--192]{OrltSandig1995}, we deduce from  \Cref{Dauge55a} that $\bar {\bm y}\in {\bm H}^{1+r}(\Omega)$. From the standard trace theorem, we have that $\bar {\bm u}\in {\bm H}^{r+1/2}(\Gamma)$.
\end{proof}

{
\begin{remark}		
In this case, the optimal control is a continuous function even for problems posed on nonconvex domains; see the second subfigure  of \Cref{Figure_04} in  \Cref{Example3} below.
\end{remark}

In order to use the Aubin-Nitsche technique to obtain error estimates in $\bm L^2(\Gamma)$ for the control variable, we are also going to study, for any given $\bm \eta\in \bm L^2(\Gamma)$, the regularity of the unique solution $\bm u_{\bm\eta}\in \bm V^{1/2}(\Gamma)$ of the problem
\[\langle \Talpha \bm u_{\bm\eta},\bm v\rangle_\Gamma = (\bm\eta,\bm v)_\Gamma\ \quad\forall \bm v\in\bm V^{1/2}(\Gamma).\]
A straightforward computation using the definitions of ${\Talpha}$, ${\bm D}$ and ${\bm E}$, gives
\begin{align*}
\langle {\Talpha} {\bm u}_{\bm\eta} -{\bm\eta},{\bm v}\rangle_\Gamma =\langle \alpha(\partial_{\bm n} {\bm y}_{\bm u_{\bm\eta}}-p_{\bm u_{\bm\eta}}{\bm n})-(\partial_{\bm n} {\bm z}({\bm y}_{\bm u_{\bm\eta}})-q({\bm y}_{\bm u_{\bm\eta}}){\bm n})-\bm\eta ,{\bm v}\rangle_\Gamma
\end{align*}
for all ${\bm v}\in {\bm V}^{1/2}(\Gamma)$.
So we have that there exists some $\lambda\in\mathbb{R}$ such that $({\bm y}_{\bm u_{\bm\eta}},p_{\bm u_{\bm\eta}})$ solves the following Neumann problem
\begin{equation*}
\left\{\begin{array}{llr}
-\Delta {\bm y}_{\bm u_{\bm\eta}}+\nabla p_{\bm u_{\bm\eta}}=0\quad &\mbox{in}\ \Omega, \\
\ \qquad \nabla\cdot {\bm y}_{\bm u_{\bm\eta}}=0\quad &\mbox{in}\  \Omega,\\
\alpha(\partial_{\bm n} {\bm y}_{\bm u_{\bm\eta}}-p_{\bm u_{\bm\eta}}{\bm n})=\partial_{\bm n} {\bm z}({\bm y}_{\bm u_{\bm\eta}})-(q({\bm y}_{\bm u_{\bm\eta}})+\lambda)\bm n+\bm\eta\quad  &\mbox{on}\ \Gamma.
\end{array} \right.
\end{equation*}
Now we can follow the reasoning of \Cref{Theorem4.1}. In this case
\begin{align*}
	{\bm e}:=\partial_{\bm n} {\bm z}({\bm y}_{\bm u_{\bm\eta}})-(q({\bm y}_{\bm u_{\bm\eta}})+\lambda)\bm n+\bm\eta\in  {\bm L}^{2}(\Gamma),\nonumber
	\end{align*}
so we are in the same situation as before, but with $t=1/2$, which leads to $\bm u_{\bm\eta}\in \bm H^1(\Gamma)$. Notice that we do not need convexity to obtain this result.

\section{FEM for the Stokes Dirichlet energy space control problem}\label{FEM_in_E}

In this section,  we consider finite element approximations to the optimal control problem (\ref{P12}). {We also briefly mention finite element approximations to the problem (\ref{P0}) in Remarks \ref{Remark6}, \ref{remark:L2__control_problem_FEM}, \ref{remark:L2_control_problem_FEM_analysis}.}

First, we assume that the finite dimensional spaces $\bm Y_h\subset\bm H^1(\Omega)$ and $W_h\subset L^2(\Omega)$ satisfy the inf-sup condition:
For each $p_h\in W_h$ there exists a $\bm y_h\in \bm Y_h$ such that
\begin{align*}
\int_\Omega p_h\nabla\cdot \bm y_h dx = \|p_h\|^2_{L^2(\Omega)}\mbox{ and }\|\bm y_h\|_{\bm H^1(\Omega)}\leqslant C\|p_h\|_{L^2(\Omega)}.
\end{align*}
It is well known that the $\mathcal P_1+$ bubble -$\mathcal P_1$ ``Mini'' element or the $\mathcal P_{k+1}-\mathcal P_k$, $k\geqslant 1$, ``Taylor-Hood'' element satisfy the inf-sup condition.

Let  ${\bm Y}_h^0:={\bm Y}_h\cap {\bm H}_0^1(\Omega)$, $W^0_h=W_h\cap L_0^2(\Omega)$ and ${\bm Y}_h(\Gamma)\subset \bm H^{1/ 2}(\Gamma)$ be the trace of ${\bm Y}_h$. Let the discrete control space be given by
\begin{eqnarray}
{\bm U}_h:=\{\bm u_h\in {\bm Y}_h(\Gamma):\ (\bm u_h\cdot \bm n, 1)_\Gamma =0\}.\label{control_set}
\end{eqnarray}

Next, we define the discrete optimization problem:
\begin{equation}\label{P12h}
\min\limits_{{\bm u}_h\in {\bm U}_h} J_h({\bm u}_h)=\frac 1 2\|{\tildeE} \bm u_h -\bm y_{d,h}\|_{\bm L^2(\Omega)}^2+\frac{\alpha}{2} ( {\tildeD}{\bm u}_h,{\bm u}_h )_\Gamma,
\end{equation}
where $\bm y_{d,h}\in\bm Y_h$ is a suitable approximation of $\bm y_d$ in the sense that $\|\bm y_{d,h}-\bm y_d\|_{\bm L^2(\Omega)}\leqslant {Ch^r}$,  and the discrete operators $\bm D_h$ and $\bm E_h$ are given below. Here, and in the rest of the paper, $r<\min\{1,\xi\}$ is the exponent obtained in  \Cref{Theorem4.1}.

We define the operators $\bm E_h: \bm H^{1/2}(\Gamma) \to {\bm Y_h}$ and $P_h:  \bm H^{1/2}(\Gamma) \to W_h^0$ by
\begin{align}\label{def_Eh_Ph}
\bm E_h \bm u = \bm y_h,\quad P_h \bm u = p_h.
\end{align}
Here $(\bm y_h, p_h)$ is the finite element approximation of $(\bm y_{\bm u}, p_{\bm u})$, i.e., $(\bm y_h, p_h)$  satisfies
\begin{align}\label{yh_ph}
\begin{split}
(\nabla\bm y_h,\nabla\bm \zeta_h) -(p_h,\nabla\cdot \bm \zeta_h)  &= 0\ \qquad\forall \bm \zeta_h\in \bm Y_h^0, \\
(\chi_h,\nabla\cdot \bm y_h) &=   0\ \qquad\forall \chi_h\in W_h^0, \\
\bm y_h &=  \projLdh\bm u\ \mbox{ on }\Gamma,
\end{split}
\end{align}
where $\projLdh\bm u$ is the $\bm L^2$ projection of $\bm u$ onto $\bm U_h$. We note for later that $ \bm Q_h $ satisfies the following standard estimate
\begin{align}\label{esti_Q}
\|\bm Q_h\bm u - \bm u\|_{\bm H^{1/2}(\Gamma)}\leqslant  {C h^r \|\bm u\|_{\bm H^{r+1/2}(\Gamma)}\ \forall \bm u\in\bm V^{1/2}(\Gamma).}
\end{align}

Next, we give the discrete approximation of the stress force on the boundary, as introduced in \cite[Section 3]{GunzburgerHou1992}. For any $\bm g\in\bm L^2(\Omega)$, we define $(\bm z_h(\bm g),q_h(\bm g))\in \bm Y_h^0\times W_h^0$ to be the unique solution of
\begin{align*}
(\nabla\bm z_h(\bm g),\nabla\bm \zeta_h) -(q_h(\bm g),\nabla\cdot \bm \zeta_h)  &= (\bm g,\bm \zeta_h)\ \qquad\forall \bm \zeta_h\in \bm Y_h^0, \\
(\chi_h,\nabla\cdot \bm z_h(\bm g))  &=  0\ \qquad\qquad\forall \chi_h\in W_h^0.
\end{align*}
For $\bm g\in \bm L^2(\Omega)$ and $\bm u\in \bm V^{1/2}(\Gamma)$, let $\bm \psi_h = \bm z_h(\bm g)+\tildeE \bm u$ and $\phi_h = q_h(\bm g) + P_h\bm u$.
We define $\bm t_h (\bm g,\bm u)\in \bm Y_h(\Gamma)$ as the approximation of the stress force on the boundary of the pair $(\bm\psi_h,\phi_h)$:
\begin{align}\label{def_Sh}
(\bm t_h( \bm g,\bm u),\bm \zeta_h)_\Gamma = (\nabla\bm \psi_h,\nabla\bm \zeta_h) -(\phi_h,\nabla\cdot \bm \zeta_h)-(\bm g,\bm \zeta_h)\ \qquad\forall \bm \zeta_h\in \bm Y_h.
\end{align}

Notice that this is exactly the concept of discrete normal derivative; see \cite{Casas_Raymond_First_SICON_2006} or, better suited for our purposes, \cite{Winkler2020}. It is also important to notice that, for $\bm v_h\in \bm Y_h(\Gamma)$, we have that
\begin{align}\label{stress_th}
	(\bm t_h(\bm g,\bm u),\bm v_h)_\Gamma = (\nabla\bm \psi_h,\nabla\bm R_h \bm v_h) -(\phi_h,\nabla\cdot \bm R_h \bm v_h)-(\bm g,\bm R_h \bm v_h)\ \qquad\forall \bm v_h\in \bm Y_h(\Gamma)
\end{align}
for any linear extension operator $\bm R_h:\bm Y_h(\Gamma)\to \bm Y_h$. For instance, $\bm R_h$ could be the discrete harmonic extension, the operator $\bm E_h$, or the zero extension.

For $\bm u\in \bm V^{1/2}(\Gamma)$ we define $\bm D_h$ as the approximation of the stress force on the boundary of the pair $(\tildeE\bm u,P_h\bm u)$:
\begin{align}\label{def_Dh}
\bm D_h \bm u = \bm t_h(\bm 0,\bm u).
\end{align}

\begin{lemma}\label{semi_norm_D_h}
	$(	\bm D_h \bm u_h, \bm u_h )_{\Gamma}^{{1/2}}$ is a  seminorm in $\bm U_h$ equivalent to the $\bm H^{1/2}(\Gamma)$ norm.
\end{lemma}
\begin{proof}
	Notice that for $\bm u_h\in \bm U_h\subset \bm V^{1/2}(\Gamma)$, using that $\tildeE\bm u_h\in \bm Y_h$ and $P_h\bm u_h\in W_h^0$, by  \eqref{def_Dh}, \eqref{def_Eh_Ph}  and \eqref{def_Sh}  we have
	\begin{align*}
	(\tildeD \bm u_h,\bm u_h)_\Gamma &=  ( \bm t_h(\bm 0,\bm u_h), \tildeE\bm u_h)_\Gamma \\
	%
	%
	%
	&=  (\nabla\tildeE \bm u_h,\nabla\tildeE \bm u_h) -(P_h\bm u_h,\nabla\cdot \tildeE \bm u_h) \\
	&=  (\nabla\tildeE \bm u_h,\nabla\tildeE \bm u_h),
	\end{align*}
	where we used  $(q_h,\nabla\cdot \tildeE \bm u_h) =   0$ for all $q_h\in W_h^0$ in the last equality.	This proves that $(	\bm D_h \bm u_h, \bm u_h )_{\Gamma}^{{1/2}}$ is a seminorm on $\bm H^{1/2}(\Gamma)$ for any $\bm u_h\in \bm U_h$.
\end{proof}

\begin{lemma}
	For every $\bm y_h\in \bm Y_h$ and $\bm v_h\in \bm U_h$, we have that
\[(\bm y_h,{\tildeE}\bm v_h) = ( - \bm t_h(\bm y_h,\bm 0) , v_h)_\Gamma,
\]
and the adjoint of the restriction of $\tildeE$ to $\bm U_h$ is given by
	\begin{align}\label{def_Eh_star}
	\tildeE^\star\bm y_h = - \bm t_h(\bm y_h,\bm 0).
	\end{align}
\end{lemma}

\begin{proof}
	We define $\bm G_h\bm y_h$ as the discrete approximation of the negative stress force on the boundary of the pair $(\bm z_h(\bm y_h),q_h(\bm y_h))$:
	\begin{align*}
	\bm G_h\bm y_h = - \bm t_h(\bm y_h,\bm 0).
	\end{align*}
	
	Consider $\bm v_h\in \bm U_h$, notice that ${\tildeE}\bm v_h\in\bm Y_h$ and by definition it equals $\bm v_h$ on the boundary. Using the definition of approximate stress force on the boundary, the facts that both $(q_h(\bm y_h),\nabla\cdot \tildeE \bm v_h) =   0$ and $(P_h\bm v_h,\nabla\cdot z_h(\bm y_h)) = 0$, and also $q_h(\bm y_h),\,P_h\bm v_h\in W_h^0$, we obtain
	\begin{align*}
	(\bm G_h\bm y_h, \bm v_h)_\Gamma &=  - (\nabla\bm z_h(\bm y_h),\nabla{\tildeE}\bm v_h)+(q_h(\bm y_h),\nabla\cdot {\tildeE}\bm v_h)+(\bm y_h,{\tildeE}\bm v_h)\\
	&=  -(P_h\bm v_h,\nabla\cdot\bm z_h(\bm y_h)) +(\bm y_h,{\tildeE}\bm v_h)\\
	&=  (\bm y_h,{\tildeE}\bm v_h)
	\end{align*}
	and the proof is complete.
\end{proof}

\begin{lemma}
	Problem \eqref{P12h} has a unique solution $\bar {\bm u}_h$.
\end{lemma}
\begin{proof}
	By \Cref{semi_norm_D_h}, it is standard to deduce that $J_h$ is coercive in $\bm U_h$. Since it is also strictly convex, problem \eqref{P12h} has a unique solution $\bar {\bm u}_h$.
\end{proof}

Following the same notation in \Cref{Stokes_enegy_space}, we define
\begin{align}\label{def_Th_wh}
{\tildeTalpha} = \alpha{\bm D_h} +\bm E_h^\star\bm E_h,\qquad \bm w_h =\bm E_h^\star\bm y_{d,h}.
\end{align}
Then the problem \eqref{P12h} can be rewritten as:
\begin{align}\label{dis_op_2}
J_h({\bm u}_h) = \frac 1 2( \bm T_h {\bm u_h}, {\bm u_h} )_\Gamma- ( \bm w_h, \bm u_h )_\Gamma + \frac{1}{2}\|\bm y_{d,h}\|_{\bm L^2(\Omega)}^2,
\end{align}
and the unique solution $\bar {\bm u}_h$ of the discrete problem satisfies the first order optimality condition
\begin{align}\label{Euler_Lagrange_pert}
( {\tildeTalpha}\bar{\bm u}_h,\bm v_h )_\Gamma=( {\tildegzero},\bm v_h)_\Gamma\quad \forall\bm v_h\in {\bm U}_h.
\end{align}

\begin{remark}\label{Remark6}
	The discretization of the problem \eqref{P0} is done in the same way. The solution of the discrete problem satisfies
	\[(\alpha \bm u_{0h} + \tildeE^\star\tildeE \bm u_{0h}, \bm v_h)_\Gamma = ({\tildegzero},\bm v_h) _\Gamma\quad \forall\bm v_h\in {\bm U}_h.\]
	Thanks to the remarkable result \cite[Theorem 5.2]{Berggren_Approximations_SINUM_2004}, the approximation of the transposition solution can be done using the discrete weak formulation given to compute $\bm E_h$.
\end{remark}

\subsection{Matrix representation of \eqref{P12h}}


{Let ${\bm Y}_h={\rm span}\{{ \bm\zeta}_n \}_{n=1}^N$ and ${\bm Y}_h^0={\rm span}\{{ \bm\zeta}_n \}_{n=1}^{N_0}$, where $ \{ \bm\zeta_n \}_n$ are nodal basis functions for ${\bm Y}_h$ ordered so that the first $N_0$ basis functions all vanish on the boundary and the remaining $ N - N_0 $ basis functions do not.} Then we have ${\bm Y}_h(\Gamma)={\rm span} \{{ \bm\zeta}_n \}_{n=N_0+1}^N$. Let $W_h={\rm span} \{\chi_n\}_{n=1}^M$, {where $ \{ \chi_n \}_n $ are nodal basis functions for $W_h$.}
For any ${\bm y}_h\in \bm Y_h$, $\bm u_h\in \bm Y_h(\Gamma)$,  ${\bm z}_h\in \bm Y_{h0}$, $p_h\in W_h$, we write
\begin{gather*}
\bm y_h(x)=\sum\limits_{n=1}^N y_n{ \bm\zeta}_n(x)=\sum\limits_{n=1}^{N_0} y_{n}{ \bm\zeta}_n(x)+\sum\limits_{n=N_0+1}^N y_n{ \bm\zeta}_n(x),\\
\bm u_h(x)=\sum\limits_{n=N_0+1}^N  u_n{ \bm\zeta}_n(x),\quad {\bm z}_h=\sum\limits_{n=1}^{N_0}{z}_n{ \bm\zeta}_n(x),\quad p_h(x)=\sum\limits_{n=1}^M p_n\chi_n(x).
\end{gather*}

We denote $\underline{ y}=( y_1, y_2,\cdots,y_N)^T$, $\underline{ y}_I=( y_1, y_2,\cdots,y_{N_0})^T$, $\underline{ u}=( u_{N_0+1}, u_{N_0+2},\cdots,u_N)^T$, $\underline{z}=(z_1,z_2,\cdots,z_{N_0})^T$, and $\underline{p}=(p_1,p_2,\cdots,p_M)^T$. Instead of imposing the condition $(p_h,1)=0$, we choose $p_h$ such that $p_M=0$ and denote $\tilde{\underline{p}} =(p_1,p_2,\cdots,p_{M-1})^T$; see  \Cref{Remark3}.

We also let $\matriz{M}$ denote the mass matrix representing the standard inner product in $\bm L^2(\Omega)$, and let
$\matriz{K}$ denote the stiffness matrix representing the vector Laplace operator on the finite
element space ${\bm Y}_h$. Additionally, $\matriz{B}$ denotes the matrix representation of the divergence operator on the involved finite element spaces ${\bm Y}_h$ and $W_h$. We have
\begin{align*}
\matriz{M}= \left[\left({\bm\zeta}_j,{ \bm\zeta}_i\right)\right]_{i,j=1}^N;\quad
\matriz{K}=\left[\left(\nabla{ \bm\zeta}_j,\nabla{ \bm\zeta}_i\right)\right]_{i,j=1}^N;\quad
\matriz{B}=-\left[\left(\chi_i,\nabla\cdot{ \bm\zeta}_j\right)\right]_{i,j=1}^{i=M,j=N}.
\end{align*}
We also use the following submatrices
\begin{align*}
\matriz{M}_{0}&= \left[\left({ \bm\zeta}_j,{ \bm\zeta}_i\right)\right]_{i,j=1}^{i=N_0,j=N}; &
\matriz{M}_{00}&= \left[\left({ \bm\zeta}_j,{ \bm\zeta}_i\right)\right]_{i,j=1}^{N_0}; \\
\matriz{M}_{\Gamma 0} &= \left[\left({ \bm\zeta}_j,{ \bm\zeta}_i\right)\right]_{i=N_0+1,j=1}^{i=N,j=N_0}; &
\matriz{M}_{\Gamma\Gamma} &= \left[\left({ \bm\zeta}_j,{ \bm\zeta}_i\right)\right]_{i,j=N_0+1}^{N};\\
\matriz{K}_{0}&= \left[\left(\nabla{ \bm\zeta}_j,\nabla{ \bm\zeta}_i\right)\right]_{i,j=1}^{i=N_0,j=N}; &
\matriz{K}_{00}&= \left[\left(\nabla{ \bm\zeta}_j,\nabla{ \bm\zeta}_i\right)\right]_{i,j=1}^{N_0}; \\
\matriz{K}_{\Gamma 0}&= \left[\left(\nabla{ \bm\zeta}_j,\nabla{ \bm\zeta}_i\right)\right]_{i=N_0+1,j=1}^{i=N,j=N_0}; &
\matriz{K}_{\Gamma \Gamma}&= \left[\left(\nabla{ \bm\zeta}_j,\nabla{ \bm\zeta}_i\right)\right]_{i,j = N_0+1}^{N};\\
\matriz{B}_0&=- \left[\left(\chi_i,\nabla\cdot{ \bm\zeta}_j\right)\right]_{i,j=1}^{i=M,j=N_0}; &
\matriz{B}_\Gamma&=- \left[\left(\chi_i,\nabla\cdot{ \bm\zeta}_j\right)\right]_{i=1,j=N_0+1}^{i=M,j=N}.
\end{align*}
Since we impose the condition $p_M=0$, instead of $\matriz{B}$, we use
\begin{align*}
\tilde{\matriz{B}}=-\left[\left(\chi_i,\nabla\cdot{ \bm\zeta}_j\right)\right]_{i,j=1}^{i={M-1},j=N}, \
\tilde{\matriz{B}}_0=-\left[\left(\chi_i,\nabla\cdot{ \bm\zeta}_j\right)\right]_{i,j=1}^{i=M-1,j=N_0},\
\tilde{\matriz{B}}_\Gamma=-\left[\left(\chi_i,\nabla\cdot{ \bm\zeta}_j\right)\right]_{i=1,j=N_0+1}^{i=M-1,j=N}.
\end{align*}

Now,  we give the implementation details for the discrete optimization problem \eqref{P12h}. By \eqref{dis_op_2}, it is equivalent to solve the following equality constrained quadratic programming problem:
\begin{align}\label{dis_op_3}
	\begin{cases}
		\min\ \dfrac{1}{2}\langle\tildeTalpha \bm u_h,\bm u_h\rangle_\Gamma- (\bm w_h,\bm u_h)_\Gamma,\\
		\bm u_h\in \bm Y_h(\Gamma),\ (\bm u_h,\bm n)_{\Gamma}=0.
	\end{cases}
\end{align}

Let $N_\Gamma = \dim \bm Y_h(\Gamma)$.  Define $\matriz{T}\in \mathbb{R}^{N_\Gamma\times N_\Gamma}$ to be the matrix representation of $\tildeTalpha$, i.e., $\underline{v}^T\matriz{T} \underline{ u}  = (\tildeTalpha \bm u_h,\bm v_h)_\Gamma$ for all $\bm u_h$ and $\bm v_h\in \bm Y_h(\Gamma)$, $\underline{w}\in \mathbb{R}^{N_\Gamma}$ to be the vector representation of $\tildegzero$, i.e., $\underline{ u}^T\underline{w}=(\bm w_h,\bm u_h)_\Gamma$ for all $\bm u_h\in \bm Y_h(\Gamma)$, and $\underline{ c}\in \mathbb{R}^{N_\Gamma}$ to be the vector representation of the constraint, i.e., $\underline{ u}^T\underline{ c}=(\bm u_h,\bm n)_\Gamma$ for all $\bm u_h\in \bm Y_h(\Gamma)$. Then the problem \eqref{dis_op_3} can be rewritten as
\begin{align}\label{dis_op_4}
	\begin{cases}
		\min\ \dfrac{1}{2}\underline{ u}^T\matriz{T} \underline{ u}-\underline{ u}^T\underline{w},\\
		\underline{ u}\in \mathbb{R}^{N_\Gamma},\quad \underline{ u}^T\underline{ c}=0.
	\end{cases}
\end{align}

Next, we show how to get the vector representations of $\underline w$ and $\matriz{T} \underline{ u}$. To do this,  we first compute $\bm E_h^\star\bm y_{d,h}$. Consider the discrete extension operator $\bm R_h\bm v_h\in \bm Y_h$ such that $\bm R_h\bm v_h=\bm v_h$ on $\Gamma$ and $\bm R_h\bm v_h = \bm 0$ in the interior nodes of $\Omega$. By the definition of $\bm w_h$ in \eqref{def_Th_wh} and using \eqref{def_Eh_star} and \eqref{stress_th} we have
\begin{align*}
	\langle  \bm w_h,\bm v_h\rangle_\Gamma &= \langle \tildeE^\star\bm y_{d,h},\bm v_h\rangle_\Gamma\\
	&=-(\nabla {\bm z}_h({\bm y}_{d,h}),\nabla {\bm R}_h\bm v_h)+(q_h({\bm y}_{d,h}),\nabla\cdot {\bm R}_h\bm v_h)+({\bm y}_{d,h},{\bm R}_h\bm v_h)\\
	 &= (-\matriz{K}_{\Gamma 0}\underline{z}_d-\tilde{\matriz{B}}^T_{\Gamma}\tilde {\underline{q}}_d+{\matriz{M}_{\Gamma}}\underline{ y}_d )\cdot \underline{ v}.
\end{align*}
\LinesNumbered
\begin{algorithm}[H]
	\caption{computation of $\underline{ w} $}
	\DontPrintSemicolon
	compute  $(\underline z_d, \tilde{\underline{q}}_d)$  by solving
	\begin{align*}
	\matriz{K}_{00}\underline{z}_d+\tilde{\matriz{B}}^T_0 \tilde{\underline{q}}_d&=\matriz{M}_{0}\underline{ y}_d,\\
	\tilde{\matriz{B}}_{0}\underline{z}&=\underline{ 0}.
	\end{align*}
	\;
	set $\underline{ w} = -\matriz{K}_{\Gamma 0}\underline{z}_d-\tilde{\matriz{B}}^T_{\Gamma}\tilde {\underline{q}}_d+{\matriz{M}_{\Gamma}}\underline{ y}_d $.\;
\end{algorithm}
\vspace{1em}

By the definiton $\bm T_h$ in \eqref{def_Th_wh}, we need to compute $( \bm E_h^\star\bm E_h \bm u_h,\bm v_h)_\Gamma $ and $( {\tildeD}{\bm u}_h,\bm v_h )_\Gamma$. Using \eqref{def_Eh_Ph}, \eqref{def_Eh_star} and \eqref{stress_th} we have
\begin{align*}
( \bm E_h^\star\bm E_h \bm u_h,\bm v_h)_\Gamma &= ( \tildeE^\star\bm y_{h},\bm v_h)_\Gamma\\
&=-(\nabla {\bm z}_h({\bm y}_h),\nabla {\bm R}_h\bm v_h)+(q_h({\bm y}_h),\nabla\cdot {\bm R}_h\bm v_h)+({\bm y}_h,{\bm R}_h\bm v_h)\\
&=(-\matriz{K}_{\Gamma 0}\underline{z}-\tilde{\matriz{B}}^T_{\Gamma}\tilde {\underline{q}}+{\matriz{M}_{\Gamma}}\underline{ y} )\cdot \underline{ v}.
\end{align*}

Now we are in the position to derive the matrix representation of the perturbed Steklov-Poincar\'e operator ${\tildeD}$. Using  \eqref{def_Dh} and \eqref{stress_th} we have
\begin{align*}
 ( {\tildeD}{\bm u}_h,\bm v_h )_\Gamma& =(\nabla {\tildeE}\bm u_h,\nabla  {\bm R_h} \bm v_h)-( P_h{\bm u}_h,\nabla\cdot  {\bm R_h} \bm v_h)\\
&=(\matriz{K}_{\Gamma} \underline{ y} +\tilde{\matriz{B}}_\Gamma^T  \tilde{\underline{p}})\cdot \underline{v}.
\end{align*}
\LinesNumbered
\begin{algorithm}[H]
	\caption{computation of $\matriz{T} \underline{ u}$}
	\DontPrintSemicolon
	compute $(\underline y_I, \tilde{\underline{p}})$ by solving
	\begin{align*}
		\matriz{K}_{00}\underline{ y}_I+\tilde{\matriz{B}}^T_0  \tilde{\underline{p}}&=-\matriz{K}_{\Gamma0}^T\underline{ u},\\
		\tilde{\matriz{B}}_{0}\underline{ y}_I&=-\tilde{\matriz{B}}_\Gamma \underline{ u}.
	\end{align*}\;
	recover $\underline{ y} = (\underline{ y}_I,\underline{ u})^T$, and then compute $(\underline z, \tilde{\underline{q}})$ by solving
	\begin{align*}
		\matriz{K}_{00}\underline{z}+\tilde{\matriz{B}}^T_0 \tilde{\underline{q}}&=\matriz{M}_{0}\underline{ y},\\
		\tilde{\matriz{B}}_{0}\underline{z}&=\underline{ 0}.
	\end{align*}\;
	set $\matriz{T} \underline{ u} =
	\alpha(\matriz{K}_{\Gamma} \underline{ y}+\tilde{\matriz{B}}_\Gamma^T\tilde{\underline{p}}) + \matriz{M}_{\Gamma}\underline{ y} -\matriz{K}_{\Gamma0}\underline{z} -\tilde{\matriz{B}}^T_{\Gamma}\tilde{\underline{q}}$.
\end{algorithm}
\vspace{1cm}
The problem \eqref{dis_op_4} can  be easily transformed into an unconstrained problem following the so-called null space method; see e.g. \cite[page 462]{Nocedal-Wright1999}. We denote $\underline{c}^T = (c_1,\ldots,c_{N_\Gamma})$ and assume, without loss of generality, that $c_1\neq 0$. The columns of the null space of $\underline{c}^T$ form the matrix $\matriz{Z}\in\mathbb{R}^{N_\Gamma\times (N_\Gamma-1)}$ such that $\matriz{Z}_{1,j} = -c_{j+1}/c_1$ and $\matriz{Z}_{i+1,j} = \delta_{i,j}$ for $1\leqslant i,j\leqslant N_{\Gamma}-1$. We solve
\begin{align*}
\begin{cases}
\min\ \dfrac{1}{2}\underline{ x}^T\matriz{Z}^T\matriz{T}\matriz{Z} \underline{ x}-\underline{ x}^T\matriz{Z}^T\underline{w}\\
\underline{ x}\in \mathbb{R}^{N_\Gamma-1}	
\end{cases}
\end{align*}
and then recover $\underline{ u} = \matriz{Z}\underline{ x}$. The Lagrange multiplier related to the constraint can also be recovered by means of
\[\lambda = \frac{\underline{ c}^T(\matriz{T}\underline{ u} -\underline{w})}{\underline{ c}^T\underline{ c}}.\]

We can also write the ``big'' optimality system. Noticing that
\begin{align*}
( \tildeE^\star(\bm y_{h}-\bm y_{d,h}),\bm v_h)_\Gamma&=-(\nabla {\bm z}_h(\bm y_{h}-\bm y_{d,h}),\nabla {\bm R}_h\bm v_h)+(q_h(\bm y_{h}-\bm y_{d,h}),\nabla\cdot {\bm R}_h\bm v_h)\\
&\quad +(\bm y_{h}-\bm y_{d,h},{\bm R}_h\bm v_h)\\
&=(\matriz{M}_{\Gamma}(\underline{ y}-\underline{y_d})-\matriz{K}_{\Gamma0}\underline{z} -\tilde{\matriz{B}}^T_{\Gamma}\tilde {\underline{q}})\cdot \underline{ v},
\end{align*}
using that $\matriz{M}_{\Gamma}\underline{ y} = \matriz{M}_{\Gamma0}\underline{ y}_I+\matriz{M}_{\Gamma\Gamma}\underline{ u}$ and $\matriz{K}_{\Gamma}\underline{ y} = \matriz{K}_{\Gamma0}\underline{ y}_I+\matriz{K}_{\Gamma\Gamma}\underline{ u}$, and taking into account that $\underline{u} = \matriz{Z}\underline{x}$, we have
\[
\left(
\begin{array}{ccccc}
\matriz{K}_{00}           & \matriz{K}_{\Gamma 0}^T \matriz{Z}        & \tilde{\matriz{B}}_0^T  &                      &                    \\
\tilde{\matriz{B}}_0      & \tilde{\matriz{B}}_\Gamma \matriz{Z}      &                    &                      &                    \\
-\matriz{M}_{00}           & -\matriz{M}_{\Gamma 0}^T \matriz{Z}        &                    & \matriz{K}_{00}           & \tilde{\matriz{B}}_0^T  \\
&                                    &                    & \tilde{\matriz{B}}_0      &                    \\
{\matriz{Z}}^T {\matriz{A}_{\Gamma 0}} &
{\matriz{Z}}^T  {\matriz{A}_{\Gamma 0}} \matriz{Z}  &
{\matriz{Z}}^T\alpha \tilde{\matriz{B}}_\Gamma^T  &
-{\matriz{Z}}^T \matriz{K}_{\Gamma 0} &
-{\matriz{Z}}^T\tilde{\matriz{B}}_\Gamma^T
\end{array}
\right)
\left(
\begin{array}{c}
\underline{ y}_I \\
\underline{ x} \\
\tilde {\underline{p}} \\
\underline{z} \\
\tilde{\underline{q}}
\end{array}
\right)=
\left(
\begin{array}{c}
\underline{ 0} \\
\underline{ 0} \\
-\matriz{M}_0\underline{ y}_{d} \\
\underline{ 0} \\
{\matriz{Z}}^T\matriz{M}_\Gamma\underline{ y}_{d}
\end{array}
\right),
\]
where
\[ { \matriz{A}_{\Gamma 0} = \alpha \matriz{K}_{\Gamma 0} + \matriz{M}_{\Gamma 0}  \:\: \mbox{and} \:\:   { \matriz{A}_{\Gamma \Gamma} = \alpha \matriz{K}_{\Gamma \Gamma} + \matriz{M}_{\Gamma \Gamma}.} } \]
The above system is not symmetric. There exist alternative symmetric formulations, at the price of the inversion of the stiffness matrix; see e.g. \cite[eq. (3.32)]{Of_Phan_Steinbach_Energy_NM_2015} for an antisymmetric version.
\begin{remark}\label{remark:L2__control_problem_FEM}
	To solve the $\bm L^2$-regularized problem, the procedure is very similar. The only difference, cf. \cite{Mateos2017}, is the computation of $\matriz{T}\underline{u}$, which is done in the following way
	\[\matriz{T} \underline{ u} =
	\alpha\matriz{S}_{\Gamma\Gamma}\underline{ u}  + \matriz{M}_{\Gamma}\underline{ y} -\matriz{K}_{\Gamma0}\underline{z} -\tilde{\matriz{B}}^T_{\Gamma}\tilde{\underline{q}},\]
	where $S_{\Gamma\Gamma}$ is the mass matrix on the boundary,
	\[\matriz{S}_{\Gamma\Gamma}={( {\bm \zeta}_j,{\bm \zeta}_i)_{\Gamma}}_{i,j=N_0+1}^{ N}.\]
	An approximation of the quantity $\lambda_0$ can be done using the Lagrange multiplier by means of $\lambda_0 = -\lambda/|\Gamma|$.
\end{remark}

\subsection{Error analysis}

First, we state the main result in this section.
\begin{theorem}\label{error_analysis}
	Let $\bar{\bm u}\in \bm V^{r+1/2}(\Gamma)$, with $r<\min\{1,\xi\}$, be the unique solution of problem (\ref{P12}) and let $\bar{\bm u}_h\in {\bm U}_h$ be the solution of \eqref{P12h}. If the conditions in  \Cref{Theorem4.1} are all fulfilled, then
	\begin{align*}
	\|\bar{\bm u}-\bar{\bm u}_h\|_{\bm H^{1/2}(\Gamma)}\leqslant C h^r\|\bar{\bm u}\|_{\bm H^{r+1/2}(\Gamma)}.
	\end{align*}
\end{theorem}

To prove \Cref{error_analysis}, we assume that the following approximation properties are satisfied (see \cite[Chapter II. Section 1.3]{GR}):
\begin{itemize}
  \item[(H1)] ({\em Approximation property of $\bm Y_h$}). There exists an operator $r_h\in\mathcal{L}(\bm H^2(\Omega),\bm Y_h)$ such that
   \begin{align*}
	   \|\bm y-r_h\bm y\|_{\bm H^1(\Omega)}\leqslant C h \|\bm y\|_{\bm H^2(\Omega)}\ \quad\forall\bm y\in \bm H^2(\Omega),
   \end{align*}
     $r_h$ preserves the boundary conditions, and
   \begin{align*}
   		\|\bm u - r_h\bm u\|_{\bm H^{1/2}(\Gamma)}\leqslant C h \|\bm u\|_{\bm H^{3/2}(\Gamma)}\ \quad\forall\bm u\in \bm H^{3/2}(\Gamma).
   \end{align*}
  \item[(H2)] ({\em Approximation property of $W_h$}). There exists an operator $S_h\in\mathcal{L}(L^2(\Omega),W_h)$ such that
    \begin{align*}
    \|p-S_h p\|_{L^2(\Omega)}\leqslant Ch \|p\|_{H^1(\Omega)}\ \quad\forall p\in H^1(\Omega).
    \end{align*}
\end{itemize}

These assumptions are satisfied by typical finite element spaces used to solve the Stokes equation, such as the $\mathcal P_1+$ bubble -$\mathcal P_1$ ``Mini'' element or the $\mathcal P_{k+1}-\mathcal P_k$, $k\geqslant 1$, ``Taylor-Hood'' element; see \cite[Chap. II, Secs. 4.1 and 4.2]{GR}, where we take $r_h$ to be the corresponding Lagrange interpolation operator and $S_h$ the $L^2(\Omega)$ projection.
}

\begin{lemma}\label{Lemma3}
	There exists a constant $C>0$ independent of $h$ such that for any $\bm g\in \bm L^2(\Omega)$ and $\bm v\in \bm V^{1/2}(\Gamma)$ we have
	\begin{align*}
	\|\bm t_h(\bm g,\bm v)\|_{\bm H^{-1/2}(\Gamma)}\leqslant C (\|\bm g\|_{\bm L^2(\Omega)}+\|\bm v\|_{\bm H^{1/2}(\Gamma)}).
	\end{align*}
	Moreover, if $\bm v\in \bm V^{r+1/2}(\Gamma)$, we have the error estimate
	\begin{align*}
	\|\bm t(\bm g,\bm v)-\bm t_h(\bm g,\bm v)\|_{\bm H^{-1/2}(\Gamma)}\leqslant C h^r (\|\bm g\|_{\bm L^2(\Omega)}+\|\bm v\|_{\bm H^{r+1/2}(\Gamma)}).
	\end{align*}
\end{lemma}
\begin{proof}
	The error estimate follows directly from \cite[Proposition 17]{GunzburgerHou1992} and  \Cref{Dauge55b,continuous_mapping}.
\end{proof}

In the next lemma, we collect the approximation properties of ${\tildeE}$, $\tildeE^\star$ and ${\tildeD}$ that will be used to obtain the final error estimate.
\begin{lemma}\label{Lemma2}
	The approximate solution operators $\bm E_h:   {\bm V}^{1/ 2}(\Gamma)\rightarrow \bm L^2(\Omega)$, $\bm E_h^\star: \bm L^2(\Omega)\to  {\bm V}^{-1/ 2}(\Gamma)$, $\bm D_h: \bm H^{1/2}(\Gamma)\to \bm H^{-1/2}(\Gamma)$ are bounded, i.e., there exists a constant $C>0$ independent of $h$ such that
	\begin{subequations}
		\begin{align}
			\|\bm E_h \bm u \|_{\bm L^2(\Omega)}&\leqslant C \|\bm u\|_{\bm H^{1/2}(\Gamma)},\label{boundness_of_E_h}\\
			\|\bm E_h^\star \bm g \|_{\bm H^{-1/2}(\Gamma)}&\leqslant C \|\bm g\|_{\bm L^2(\Omega)},\label{boundness_of_E_h_star}\\
			\|\bm D_h \bm u\|_{\bm H^{-1/2}(\Gamma)}&\leqslant C \|\bm u\|_{\bm H^{1/2}(\Gamma)}.\label{boundness_of_D_h}
		\end{align}
	\end{subequations}
Moreover, for $\bm u\in    {\bm V}^{r+1/2}(\Gamma)$ and $\bm g\in \bm H^r(\Omega)$, the following error estimates hold:
	\begin{subequations}
		\begin{align}
		\|{\bm E}\bm u-{\tildeE}\bm u\|_{\bm L^2(\Omega)} &\leqslant  Ch^{r}\|\bm u\|_{ \bm H^{r+1/2}(\Gamma)},\label{E_error}\\
		\|{\bm E^\star}\bm g-\bm E_h^\star\bm g\|_{\bm H^{-1/2}(\Gamma)}&\leqslant  Ch^{r}\|\bm g\|_{ \bm H^{r}\bm(\Omega)},\label{Estar_error}\\
		\|\bm D \bm u - \bm D_h \bm u\|_{\bm H^{-1/2}(\Gamma)} &\leqslant C h^r \|\bm u\|_{\bm H^{r+1/2}(\Gamma)}.\label{error_boundness_of_D_h}
		\end{align}	
	\end{subequations}

\end{lemma}
\begin{proof}
The boundness of $\bm E_h$ and  the  approximation error follow directly from \cite[Theorem 15]{GunzburgerHou1992} and the continuous embedding $\bm H^1(\Omega)\hookrightarrow\bm L^2(\Omega)$. The remaining estimates can be easily obtained by \Cref{Lemma3}, \eqref{def_Eh_star}, and  \eqref{def_Dh}.
\end{proof}

Next, we introduce the following auxiliary problem: find $\widehat{\bm u}_h\in {\bm U}_h$ such that
\begin{align}\label{Euler_Lagrange_h}
( {\Talpha}\widehat {\bm u}_h,\bm v_h)_\Gamma=( {\gzero},\bm v_h)_\Gamma\quad \forall\bm v_h\in {\bm U}_h,
\end{align}
where ${\gzero}={\bm E}^\star{\bm y}_d\in {\bm V}^{-{1/ 2}}(\Gamma)$.
\begin{lemma}\label{Lemma1}
	Let $\bar{\bm u}\in \bm V^{r+1/2}(\Gamma)$, with $r<\min\{1,\xi\}$, be the unique solution of problem (\ref{P12}) and $\widehat{\bm u}_h\in {\bm U}_h$ be the solution of  \eqref{Euler_Lagrange_h}. If the conditions in  \Cref{Theorem4.1} are all fulfilled, then
	\begin{align}\label{error_u_uhat}
	\|\bar{\bm u}-\widehat{\bm u}_h\|_{\bm H^{1/2}(\Gamma)}\leqslant Ch^{r}\|\bar{\bm u}\|_{\bm H^{{1/2}+r}(\Gamma)}.
	\end{align}
\end{lemma}
\begin{proof}
	First, by \eqref{Euler_Lagrange}, \eqref{Euler_Lagrange_h}, and ${\bm U}_h\subset \bm V^{1/2}(\Gamma)$, we have
	\begin{align}\label{orthogonal_propertity}
	( {\Talpha}(\bar{\bm u}-\widehat {\bm u}_h),\bm v_h)_\Gamma=0\quad \forall\bm v_h\in {\bm U}_h.	
	\end{align}
	Next, by \Cref{property_T}, we know that ${\Talpha}$ is $\bm V^{1/2}(\Gamma)$-elliptic and continuous. For any $\bm u_h^\star\in \bm U_h$, the error estimate follows in a standard way:
	\begin{align*}
	c\|\bar{\bm u}-\widehat{\bm u}_h\|_{\bm H^{1/2}(\Gamma)}^2 &\leqslant ( {\Talpha} (\bar{\bm u}-\widehat{\bm u}_h),\bar{\bm u}-\widehat{\bm u}_h)_\Gamma \\
	&=
	( {\Talpha} (\bar{\bm u}-\widehat{\bm u}_h),\bar{\bm u}-\bm u_h^\star)_\Gamma\\
	& \leqslant \|{\Talpha} (\bar{\bm u}-\widehat{\bm u}_h)\|_{\bm V^{-1/2}(\Gamma)} \|\bar{\bm u}-\bm u_h^\star\|_{\bm H^{1/2}(\Gamma)} \\
	& \leqslant C \|\bar{\bm u}-\widehat{\bm u}_h\|_{\bm H^{1/2}(\Gamma)} \|\bar{\bm u}-\bm u_h^\star\|_{\bm H^{1/2}(\Gamma)}.
	\end{align*}
	Therefore, there exists $C>0$ such that
	\begin{align*}
	\|\bar{\bm u}-\widehat{\bm u}_h\|_{\bm H^{1/2}(\Gamma)}\leqslant C\inf_{\bm u_h^\star\in \bm U_h} \|\bar{\bm u}-\bm u_h^\star\|_{\bm H^{1/2}(\Gamma)}.
	\end{align*}
	The result follows by interpolation (see e.g. \cite[Theorem (14.3.3)]{Brenner_FEMBook_2008}), where we take $\bm u_h^\star=r_h \bm u$ from (H1) and use the regularity of $\bar{\bm u}$ stated in  \Cref{Theorem4.1}.
	
\end{proof}

Now we give the proof of \Cref{error_analysis}.

\begin{proof}[Proof of \Cref{error_analysis}]
  Due to \Cref{Lemma1}, it is enough to obtain the error estimate for $\|\bar{\bm u}_h-\widehat{\bm u}_h\|_{\bm H^{1/2}(\Gamma)}$.

  By the definition of $\bm T_h$ in \eqref{def_Th_wh} and \Cref{semi_norm_D_h}, we know that $\bm T_h$ is coercive on $\bm U_h$.  By the first order conditions satisfied by $\widehat{\bm u}_h$ and $\bar{\bm u}_h$ in \eqref{Euler_Lagrange_h} and \eqref{Euler_Lagrange_pert} {and by} Young's inequality, we know that there exists a constant $\kappa$ independent of $h$ such that
  \begin{align*}
    \kappa \|\bar{\bm u}_h-\widehat{\bm u}_h\|_{\bm H^{1/2}(\Gamma)}^2 & \leqslant ( {\tildeTalpha} (\bar{\bm u}_h-\widehat{\bm u}_h), \bar{\bm u}_h-\widehat{\bm u}_h)_\Gamma\\
    &=  (\bm w_h-\bm w,\bar{\bm u}_h-\widehat{\bm u}_h\rangle_\Gamma +
    \langle({\Talpha}-{\tildeTalpha})\widehat{\bm u}_h, \bar{\bm u}_h-\widehat{\bm u}_h)_\Gamma\\
     &\leqslant \frac{2}{\kappa}\|\bm w_h-\bm w\|_{\bm H^{-1/2}(\Gamma)}^2
    +\frac{2}{\kappa}\|({\Talpha}-{\tildeTalpha})\widehat{\bm u}_h\|_{\bm H^{-1/2}(\Gamma)}^2
    +\frac{\kappa}{4}\|\bar{\bm u}_h-\widehat{\bm u}_h\|_{\bm H^{1/2}(\Gamma)}^2.
  \end{align*}
Hence, by the definitions of $\bm T$ and $\bm T_h$ in \eqref{def_T} and \eqref{def_Th_wh} we have
\begin{align*}
	\hspace{1em}&\hspace{-1em}\frac{3}{4}\kappa\|\bar{\bm u}_h-\widehat{\bm u}_h\|_{\bm H^{1/2}(\Gamma)}^2\\
	&\leqslant \frac{2}{\kappa}\|\bm w_h-\bm w\|_{\bm H^{-1/2}(\Gamma)}^2
	+\frac{2}{\kappa}\|({\Talpha}-{\tildeTalpha})\widehat{\bm u}_h\|_{\bm H^{-1/2}(\Gamma)}^2\\
	& \leqslant  C\left(\|\bm w_h-\bm w\|_{\bm H^{-1/2}(\Gamma)} + \| ({\bm D}-{\tildeD})\widehat{\bm u}_h\|_{\bm H^{-1/2}(\Gamma)} +  \|({\bm E^\star\bm E}-\tildeE^\star\tildeE)\widehat{\bm u}_h\|_{\bm H^{-1/2}(\Gamma)}\right)^2\\
	&= C\left(S_1 + S_2 + S_3\right)^2.
\end{align*}

For the first term $S_1$, using   the approximation properties of $\bm y_{d,h}$, \eqref{def_T}, \eqref{def_Th_wh}, \eqref{Estar_error} and \eqref{boundness_of_E_h_star}, we get
  \begin{align*}
	 S_1 &= \|\bm w_h-\bm w\|_{\bm H^{-1/2}(\Gamma)}\\
	 &  = \|\tildeE^\star \bm y_{d,h} -\bm E^\star \bm y_d \|_{\bm H^{-1/2}(\Gamma)}\\
	 &\leqslant \| (\tildeE^\star -\bm E^\star) \bm y_{d}\|_{\bm H^{-1/2}(\Gamma)}
	 +\|\bm E_h^\star(\bm y_{d,h}- \bm y_d)\|_{\bm H^{-1/2}(\Gamma)}\\
	  &\leqslant C h^r.
 \end{align*}

For the second term $S_2$, by the definition of $\bm D_h$ in \eqref{def_Dh} we know that $\bm D_h\bm Q_h = \bm D_h$, where $\bm Q_h$ is the $L^2$ projection. We have
\begin{align*}
  S_2 &= \| ({\bm D}-{\tildeD})\widehat{\bm u}_h \|_{\bm H^{-1/2}(\Gamma)} \\
  &\leqslant
  \| {\bm D} (\widehat{\bm u}_h -\bar{\bm u})\|_{\bm H^{-1/2}(\Gamma)}
   +  \|\bm D\bar{\bm u} -{\tildeD}\bm Q_h\bar{\bm u}\|_{\bm H^{-1/2}(\Gamma)} + \| ({\tildeD}(\bm Q_h\bar{\bm u} -\widehat{\bm u}_h) \|_{\bm H^{-1/2}(\Gamma)}\\
  &\leqslant  C \|\widehat{\bm u}_h -\bar{\bm u}\|_{\bm H^{1/2}(\Gamma)} + \|\bm D\bar{\bm u} -{\tildeD}\bar{\bm u}\|_{\bm H^{-1/2}(\Gamma)} + C\|\bm Q_h\bar{\bm u} -\widehat{\bm u}_h\|_{\bm H^{1/2}(\Gamma)},
  \end{align*}
  where we used \eqref{boundness_of_D} and \eqref{boundness_of_D_h} in the last inequality. Next, by \eqref{error_u_uhat}, \eqref{error_boundness_of_D_h}, and \eqref{esti_Q} we have
  \begin{align*}
  S_2 & \leqslant  C \|\widehat{\bm u}_h -\bar{\bm u}\|_{\bm H^{1/2}(\Gamma)} +
\|\bm D\bar{\bm u} -{\tildeD}\bar{\bm u}\|_{\bm H^{-1/2}(\Gamma)}+ C\|\bar{\bm u} -\widehat{\bm u}_h\|_{\bm H^{1/2}(\Gamma)} +   C\|\bm Q_h\bar{\bm u} -\bar{\bm u}\|_{\bm H^{1/2}(\Gamma)}\\
  &\leqslant  C h^r\|\bar{\bm u}\|_{\bm H^{r+1/2}(\Gamma)}.
\end{align*}

Next,  for the term $S_3$ we proceed similarly to $S_2$.  Using the fact that $\bm E_h\bm Q_h = \bm E_h$, we have 
\begin{align*}
  S_3 & = \|({\bm E^\star\bm E}-\tildeE^\star\tildeE)\widehat{\bm u}_h\| _{\bm H^{-1/2}(\Gamma)}\\
  & \leqslant
  \|{\bm E^\star\bm E}(\widehat{\bm u}_h -\bar{\bm u})\|_{\bm H^{-1/2}(\Gamma)}
   +  \|(\bm E^\star -\tildeE^\star)\bm E\bar{\bm u}\|_{\bm H^{-1/2}(\Gamma)}\\
   &\quad +  \|\tildeE^\star (\bm E\bar{\bm u} -\tildeE\bm Q_h\bar{\bm u})\|_{\bm H^{-1/2}(\Gamma)} + \| (\tildeE^\star\tildeE(\bm Q_h\bar{\bm u} -\widehat{\bm u}_h) \|_{\bm H^{-1/2}(\Gamma)}\\
   &\leqslant  C \|\widehat{\bm u}_h -\bar {\bm u}\|_{\bm H^{1/2}(\Gamma)} + C h^r \|\bm E\bar{\bm u}\|_{\bm H^{r+1}(\Omega)}+ C \|\bm E\bar{\bm u} -\tildeE\bar{\bm u}\|_{\bm L^{2}(\Omega)} + C\|\bm Q_h\bar{\bm u} -\widehat{\bm u}_h\|_{\bm H^{1/2}(\Gamma)},
   \end{align*}
   where we used \eqref{boundness_of_EstarE}, \eqref{Estar_error}, \eqref{boundness_of_E_h_star},  and \eqref{boundness_of_E_h} in the last inequality. By
   \eqref{error_u_uhat}, \eqref{E_error} and \eqref{esti_Q} we have
   \begin{align*}
    S_3 &\leqslant  C \|\widehat{\bm u}_h -\bar {\bm u}\|_{\bm H^{1/2}(\Gamma)} + C h^r \|\bm E\bar{\bm u}\|_{\bm H^{r+1}(\Omega)} + C \|\bm E\bar{\bm u} -\tildeE\bar{\bm u}\|_{\bm L^{2}(\Omega)} + C\|\bm Q_h\bar{\bm u} -\bar{\bm u}\|_{\bm H^{1/2}(\Gamma)}\\
   &\leqslant  C h^r\|\bar{\bm u}\|_{\bm H^{r+1/2}(\Gamma)}.
\end{align*}
Collecting all the estimates completes the proof.
\end{proof}

\begin{remark}
 The application of the Aubin-Nitsche technique to the intermediate problem leads easily to
       \[\|\bar{\bm u}-\widehat{\bm u}_h\|_{\bm L^{2}(\Gamma)}\leqslant Ch^{r+1/2}\|\bm u\|_{\bm H^{r+1/2}(\Gamma)}.\]
       However, using this to obtain error estimates in $\bm L^2(\Gamma)$ for $\bar{\bm u}_h$ is not immediate because {$\bar{\bm u}_h$} satisfies a problem with a perturbed operator and perturbed second member. Following \cite[Remark 26.1]{Ciarlet91}, the error would be of the same order as
       \[\|(\Talpha-\tildeTalpha)\bar{\bm u}_h\|_{\bm H^{-1/2}(\Gamma)}+\|{\gzero}-\tildegzero\|_{\bm H^{-1/2}(\Gamma)}.\]

       Using the improved error estimate for the discrete approximation of the stress force on the boundary for regular solutions in \cite[Proposition 17]{GunzburgerHou1992}, we find that the convergence order $r+1/2$ for those terms can be achieved under the following two assumptions: first, that $\bm y_d\in \bm H^{r-1/2}(\Omega)$, which is quite reasonable; but also that $\bar{\bm u}_h$ is bounded in $\bm H^{1+r}(\Gamma)$. But this second assumption {requires a higher} regularity of the optimal solution; {in such a case} the order of convergence in $\bm H^{1/2}(\Gamma)$ would be increased by another $1/2$.

       In numerical experiments, this is the behavior usually observed {with the ``Mini'' finite element}: order $3/2$ in $\bm H^{1/2}(\Gamma)$ and order $2$ in $\bm L^{2}(\Gamma)$.

\end{remark}

\begin{remark}\label{remark:L2_control_problem_FEM_analysis} Although the discretizations of the $\bm L^2$ regularized problem and the $\bm H^{1/2}$ regularized problem are very similar, the error analysis performed for the second case cannot be carried out for the first because of the lack of regularity of the solution $\bm u_0\in \bm H^{s}(\Gamma)$ for $0\leqslant s<s^\star$, where $s^\star=\min\{1/2,\xi-1/2\}$.

    Using the general discretization error estimate of \cite[Theorem 3.2]{APel_Mateos_Pfefferer_Arnd_DBC_MCRF_2018}, we see that the error is bounded by the best approximation error in the space, the error related to the discretization of the state equation, and the error related to the discrete approximation of the stress force on the boundary. While we have no results for the last two ones, the first one is determined by the Sobolev exponent $s$, so one cannot expect more than $h^s$ for the error.

\end{remark}

\section{Numerical experiments}
\label{Numer_ex}
\setcounter{equation}{0}
In this section we carry out some numerical experiments to {compare the solutions of the two control problems (\ref{P0}) and (\ref{P12}), and also illustrate how the convergence orders can vary due to the shape of the domain and the problem data.} We present two examples in a square domain, the first one having a very regular solution, and one example in an L-shaped domain. We discretize each problem using the ``Mini'' finite element \cite{Koko2018} and a family of meshes of size $h_i=2^{-i}\sqrt{2}$ obtained by regular refinement of an initial coarse mesh of size $h_0=\sqrt{2}$. {For one problem, we also discretize using Taylor-Hood elements.} Since we do not have the exact solution, we compare the obtained solutions for $i=2,\ldots,I-2$ with the reference solution obtained for $i=I$, where $I=9$ for the square (a mesh with $2\times 2^{2\times 9} = 524288$ elements) and $I=8$ for the L-shaped domain (a mesh with $6\times 2^{2\times 8} =  393216$ elements). For $i=2,\ldots,I-2$ we can solve the ``big'' optimality system using {\sc Matlab}'s \texttt{mldivide}. For $i=I$ we run out of memory and solve the reduced optimality system using {\sc Matlab}'s \texttt{pcg}.

Let $\bm E_h=\bm u-\bm u_h$, we report the $\bm L^2(\Gamma)$-norm error and the $\bm H^{1/2}(\Gamma)$-seminorm error, both computed using the equivalent mesh-independent discrete norms obtained in \cite{Casas-Raymond2006b}.
\begin{example}\label{Example1}
	We consider the unit square domain $\Omega = (0,1)^2$ and set the regularization parameter $\alpha = 1.0e-3$. We choose the forcing ${\bm f}=(1,1)$, and for the target state we choose the large vortex given in \cite{John-Linke2017},
	\begin{eqnarray*}
		{\bm y}_d &=& 200\times [x_1^2(1-x_1)^2x_2(1-x_2)(1-2x_2);
		 - x_1(1-x_1)(1-2x_1)x_2^2(1-x_2)^2],
	\end{eqnarray*}
	see the left of \Cref{Figure_01}. For a related example using tangential boundary control, see \cite{GongHuMateosSinglerZhang2018}. The data size in terms of the tracking functional can be measured as ${\bm F}(\bm 0) = 0.302339$.
	Notice that $\nabla\cdot\bm y_d = 0$ and $\bm y_d=0$ on $\Gamma$, but it cannot be the solution of the Stokes problem with data $\bm f = (1,1)$ since $\bm f+\Delta\bm y_d$ is not a conservative field.
	
	For the $\bm H^{1/2}(\Gamma)$ regularization, we obtain a value for the tracking term of ${\bm F}(\bar{\bm u}) = 0.112264$, while for the $\bm L^2(\Gamma)$ regularization we obtain a slightly smaller value $\bm F(\bm u_0) = 0.111576$. A graph of the state,  the optimal control in the energy space,  and the solution of the $\bm L^2(\Gamma)$ regularized problem can be found in  \Cref{Figure_02}. In this case, $\bm u_0$ is a continuous function. Numerically, we find that $|q_0(x_j)+\lambda_0|<3\times 10^{-8}$ for all four corners $x_j$.	
	\begin{figure}
		\centering		
		\includegraphics[width=.32\textwidth]{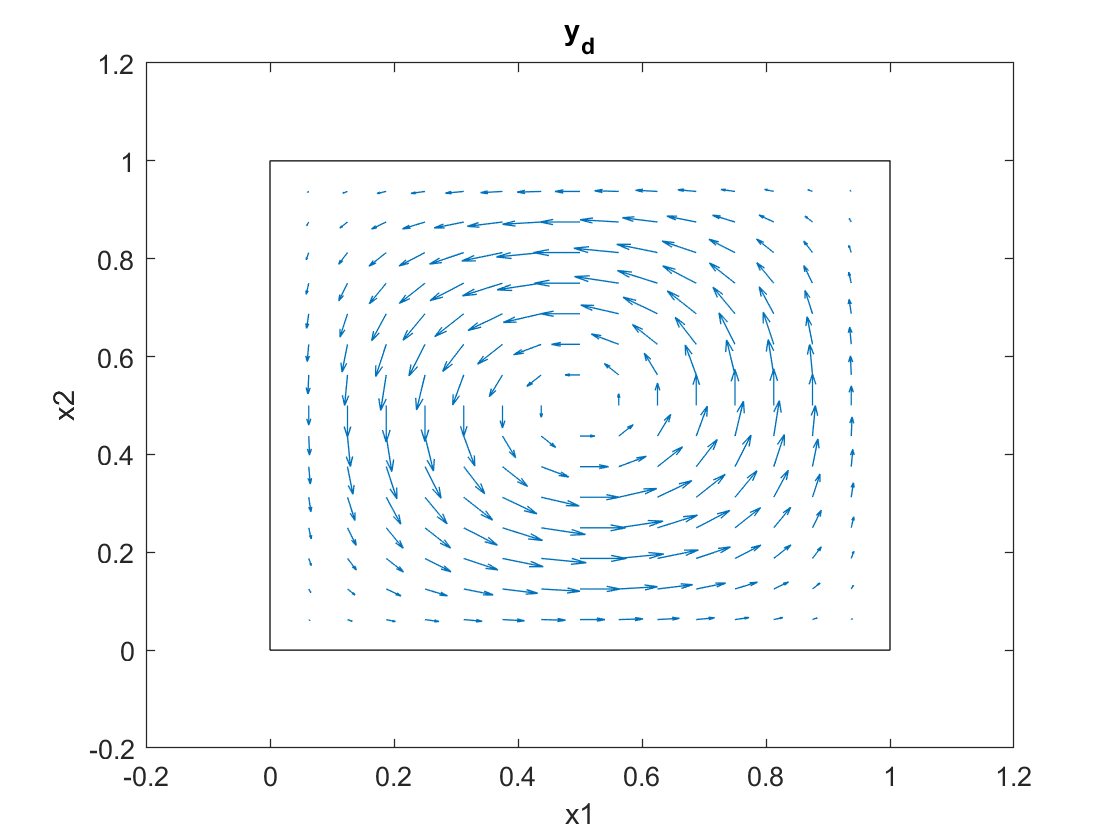}
		\includegraphics[width=.32\textwidth]{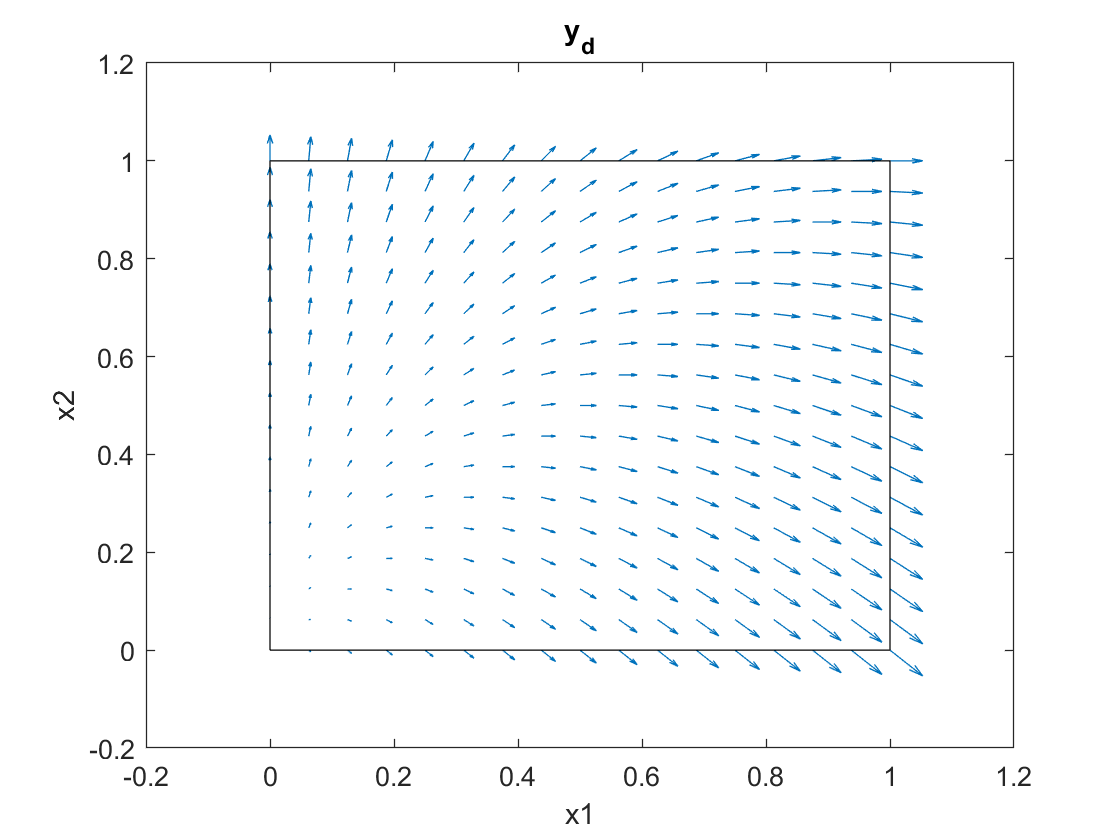}
		\includegraphics[width=.32\textwidth]{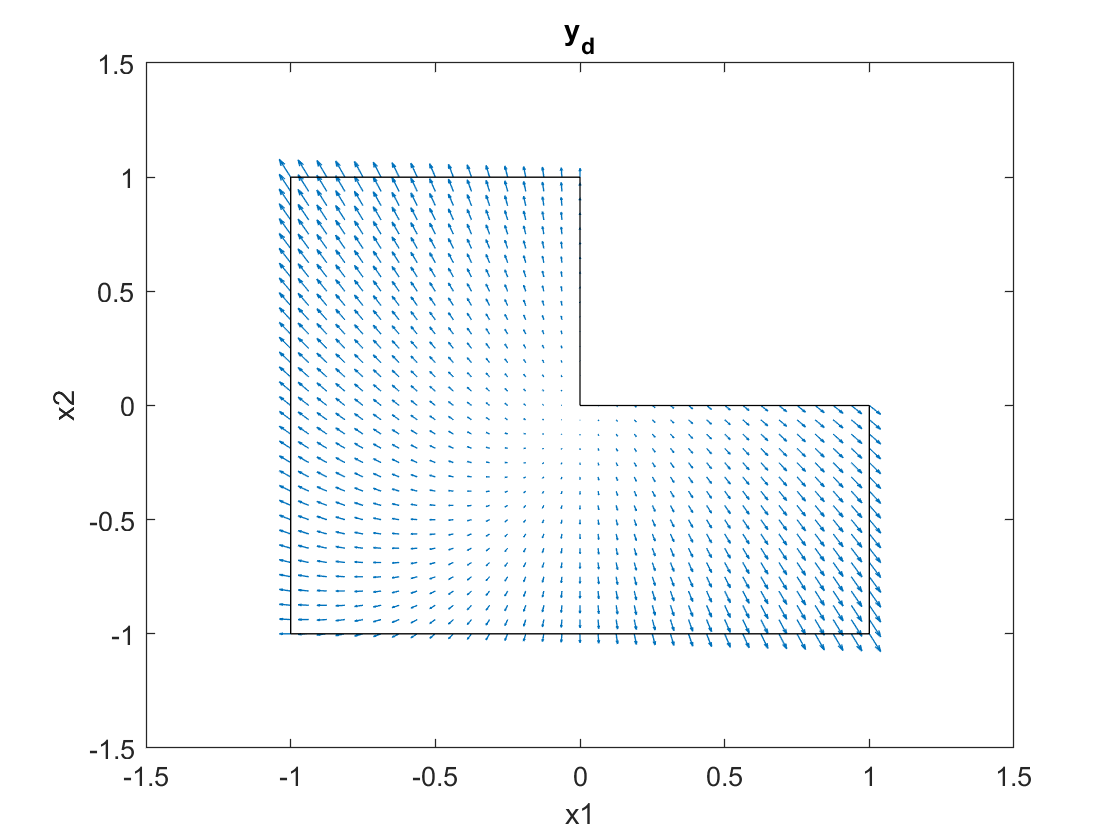}
		\caption{Left is the target of \Cref{Example1}, middle is the target of \Cref{Example2}, right is the target of \Cref{Example3}.}\label{Figure_01}
	\end{figure}

	\begin{figure}
	\centering		
	\includegraphics[width=.24\textwidth]{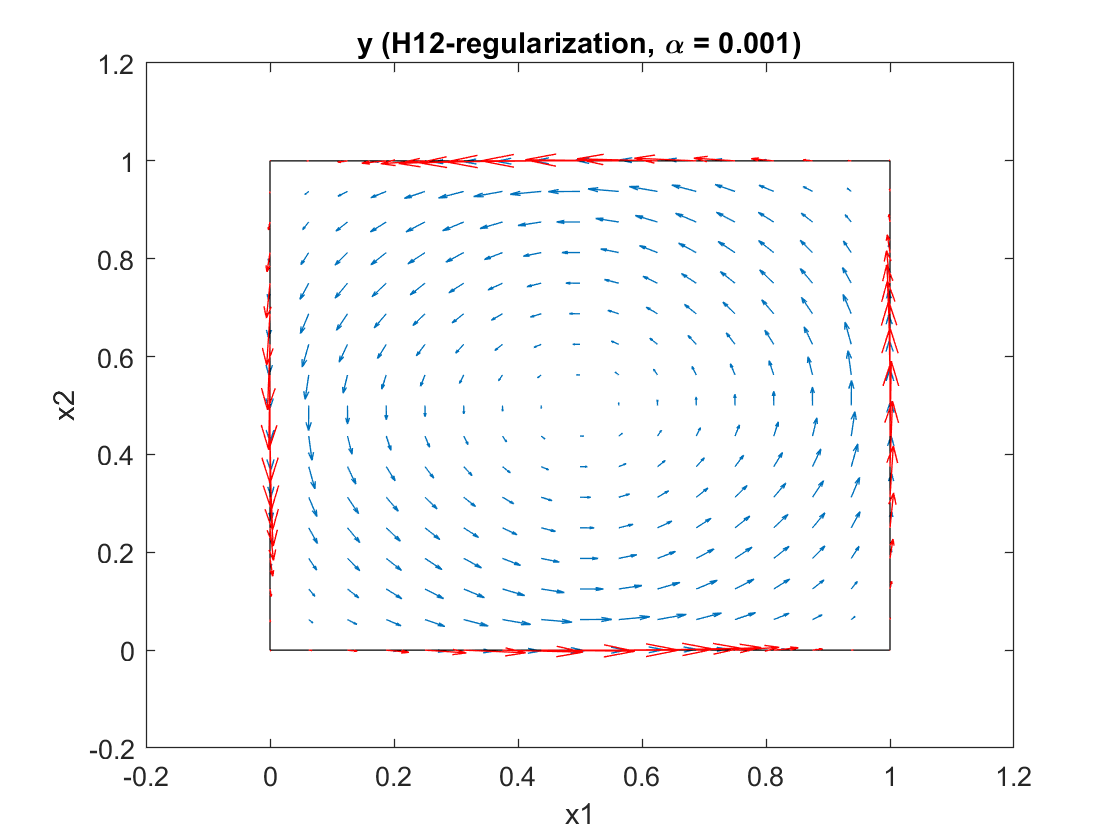}
	\includegraphics[width=.24\textwidth]{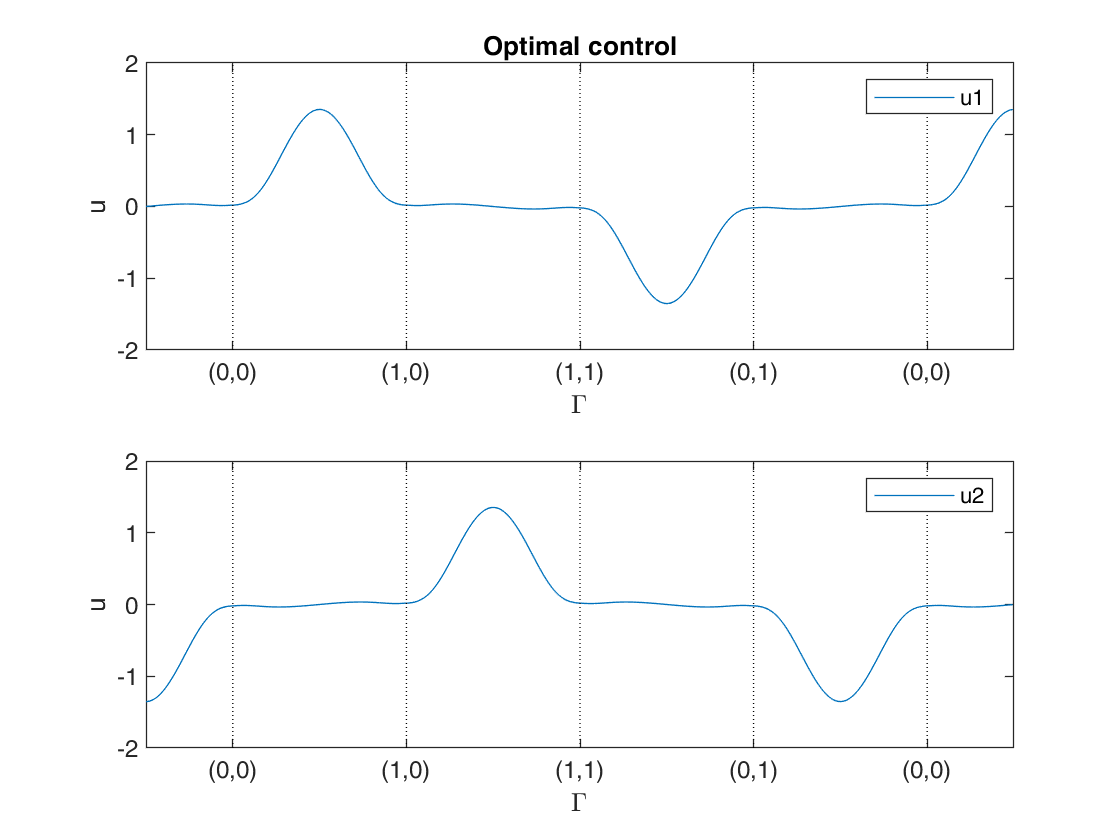}
		\includegraphics[width=.24\textwidth]{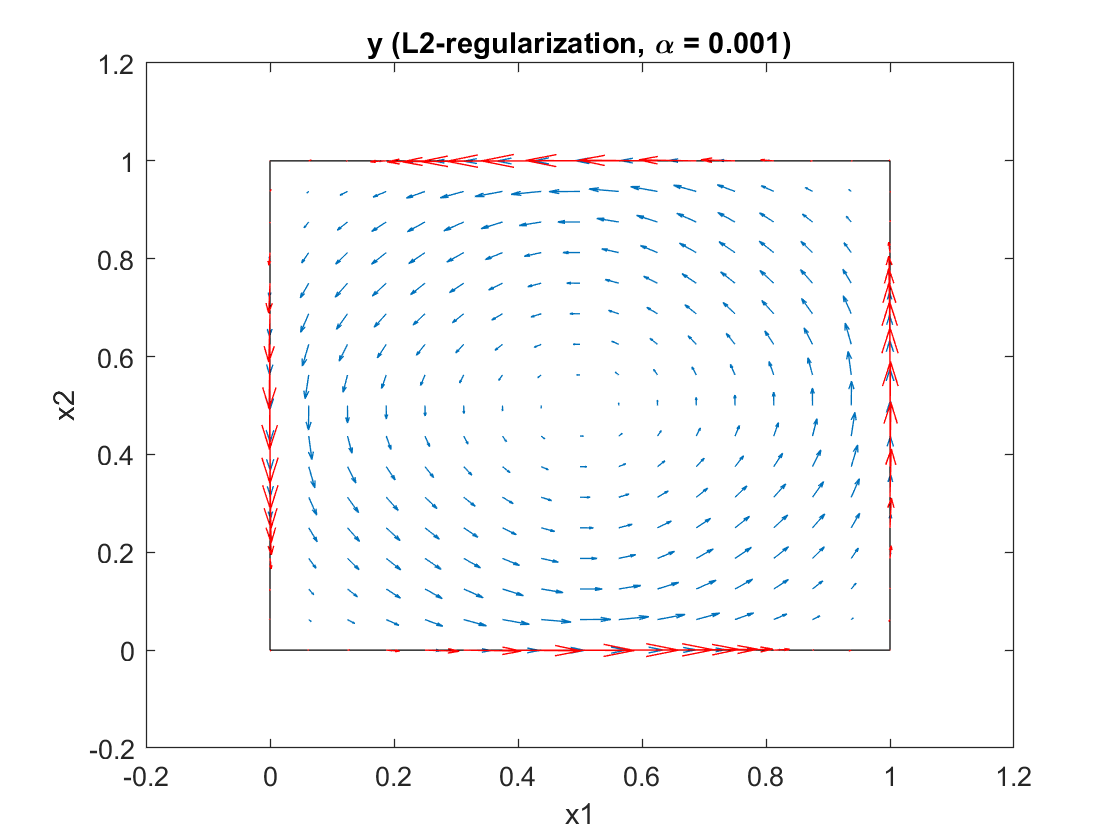}
		\includegraphics[width=.24\textwidth]{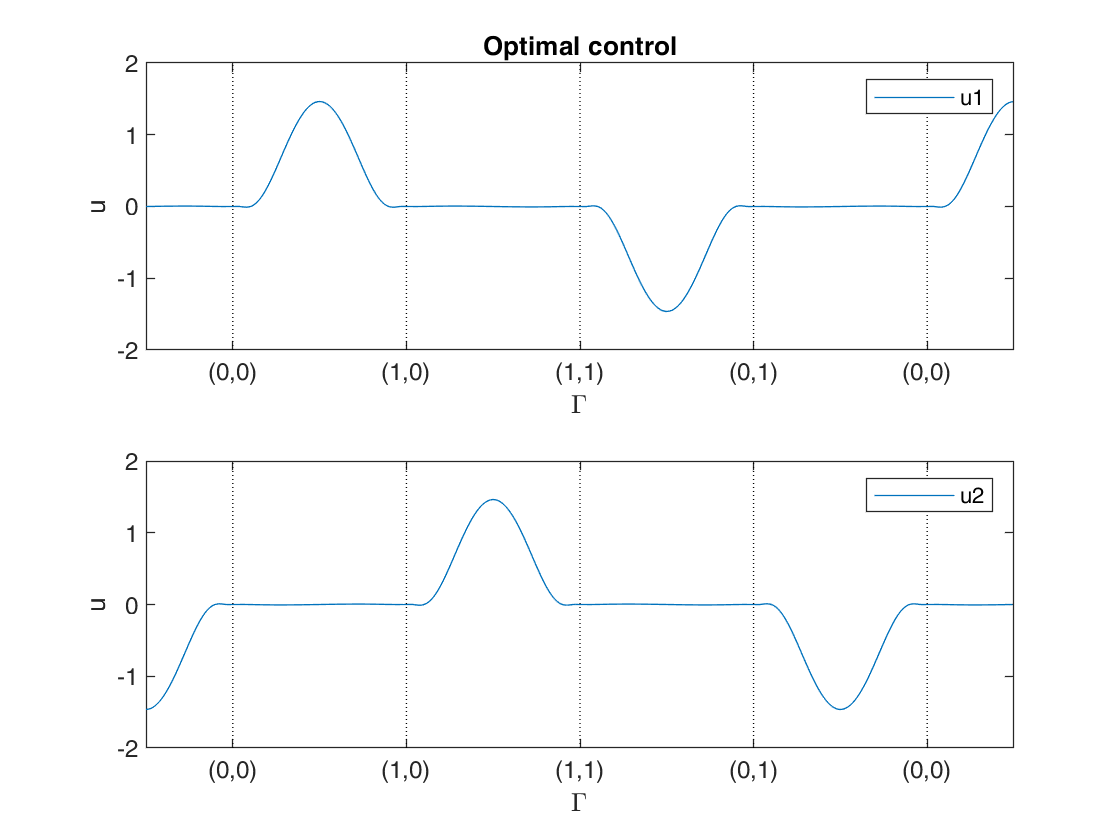}
		\caption{Solution of \Cref{Example1}:   The first two subfigures are  for  $\bm H^{1/2}(\Gamma)$ regularization, the last two subfigures are for  $\bm L^2(\Gamma)$ regularization.}\label{Figure_02}
	\end{figure}

	The value of the singular exponent for this domain is $\xi = 2.740$; see \cite[Table 1]{Dauge1989}. This means that the exponent giving the order of convergence of the energy regularized problem in the $\bm H^{1/2}(\Gamma)$-norm is  $r\approx 1$ and the exponent giving the best possible order of convergence of the $\bm L^2(\Gamma)$ regularized problem in the $\bm L^2(\Gamma)$ norm is $s\approx 0.5$.
	We obtain the results summarized in  \Cref{Table01} for the optimal control problem with $\bm H^{1/2}$ regularization  and with $\bm L^{2}(\Gamma)$ regularization. In this case the solution is very regular, the results are similar for both approaches and better than predicted by the general theory. This high regularity can also be noticed in the orders of convergence found for the other variables using higher order {Taylor-Hood} elements; see  \Cref{Table02}.
	
		\begin{table}
		\centering
				\begin{tabular}{cc|c|c|c|c|c|c|c|c}
				\Xhline{1pt}

				\multirow{2}{*}{}
				&\multirow{2}{*}{$i$}	
				&\multicolumn{4}{c|}{$\bm H^{1/2}$ regularization}	
				&\multicolumn{4}{c}{$\bm L^{2}$ regularization}	\\
				\cline{3-10}
				& &$\|\bm E_h\|_{\bm H^{1/2}(\Gamma)}$ &Rate
				&$\|\bm E_h\|_{\bm L^2(\Gamma)}$ &Rate
				&$\|\bm E_h\|_{\bm H^{1/2}(\Gamma)}$ &Rate
				&$\|\bm E_h\|_{\bm L^2(\Gamma)}$ &Rate
				\\
				\cline{1-10}
				\multirow{5}{*}{}
				&	$2$	&	 4.93E+0	&	-	    &	  8.37E-01	&	-	    &	 6.17E+0	&	-	      & 9.78E-01   &-                  \\
				&	$3$	&	1.62E+0	    &	1.61	&	  2.56E-01	&	1.71	&	 2.01E+0	&	1.62	  & 3.03E-01   &1.69                \\
				&	$4$	&	 4.82E-01	&	1.75	&	  6.80E-02	&	1.91	&	 6.37E-01	&	1.65	  & 8.00E-02   &1.92           \\
				&	$5$	&	 1.39E-01	&	1.79	&	  1.75E-02	&	1.96	&	 1.87E-01	&	1.77	  & 2.01E-02   &1.93            \\
				&	$6$	&	4.07E-02	&	1.78	&	  4.37E-03	&	2.00	&	 5.54E-02	&	1.75	  & 5.31E-03   &1.98             \\

				\Xhline{1pt}

			\end{tabular}
		\caption{Errors and experimental order of convergence for  \Cref{Example1}.}\label{Table01}
	\end{table}

	\begin{table}
		\centering
	{
		\begin{tabular}{c|c|c|c|c|c|c|c}
			\Xhline{1pt}

			\multirow{2}{*}{}
			&\multirow{2}{*}{$i$}	
			&\multicolumn{2}{c|}{$\|\bm y-\bm y_h\|_{\bm L^2(\Omega)}$}	
			&\multicolumn{2}{c|}{$\|\bm u-\bm u_h\|_{\bm H^{1/2}(\Gamma)}$}	
			&\multicolumn{2}{c}{$\|\bm z-\bm z_h\|_{\bm L^2(\Omega)}$}	\\
			\cline{3-8}
			& &Error &Rate
			&Error &Rate
			&Error &Rate
			\\
			\cline{1-8}
			\multirow{5}{*}{ $\mathcal P_2-\mathcal P_1$}
			&	$1$	&	1.73E-03	&	-	&	  2.36E-02	&	-	&	 2.03E-03	&	-	 \\
			&	$2$	&	 2.76E-04	&	2.65	&	  8.14E-03	&	1.54	&	 3.79E-04	&	2.42	 \\
			&	$3$	&	 3.69E-05	&	2.90	&	  2.20E-03	&	1.89	&	 5.13E-05	&	2.89	 \\
			&	$4$	&	 5.12E-06	&	2.85	&	  6.16E-04	&	1.83	&	 6.61E-06	&	2.95	 \\
			&	$5$	&	 7.36E-07	&	2.80	&	  1.81E-04	&	1.76	&	 1.81E-07	&	2.98	 \\

			\cline{1-8}
			\multirow{5}{*}{ $\mathcal P_3-\mathcal P_2$}
			&	$1$	&	4.54E-04	&	-	&	  9.56E-02	&	-	&	     6.28E-03	&	-	 \\
		&	$2$	&	 4.17E-05	&	3.45	&	  1.70E-03	&2.49	&	 5.67E-04	&	3.47	 \\
		&	$3$	&	 4.39E-06	&	3.25	&	 3.98E-03	&	2.10	&	 4.45E-05	&	3.67	 \\
		&	$4$	&	6.50E-07	&	2.75	&	  1.26E-04	&	1.65	&	 3.75E-06	&	3.57	 \\
		&	$5$	&	 9.61E-08	&	2.76	&	 3.85E-04	&	1.72	&	 3.20E-07	&	3.55	 \\

			\Xhline{1pt}

		\end{tabular}
	}
	\caption{Errors and experimental order of convergence for the state and adjoint state for  \Cref{Example1}.}\label{Table02}
\end{table}

\end{example}

\begin{example}\label{Example2}
	Set $\Omega = (0,1)^2$, $\alpha = 1$, $\bm f = \bm 0$ and $\bm y_d = (x_1;x_2-x_1)$. The data size is ${\bm F}(\bm 0) = 0.25$, and the target does not belong to $\bm V^0(\Omega)$. A graph of the target field is sketched in the middle of \Cref{Figure_01}.
	
	For the energy regularization, we find ${\bm F}(\bar{\bm u}) = 0.117607$; see the first two subfigures of \Cref{Figure_03}. For the $\bm L^2(\Gamma)$-regularized problem, we have that ${\bm F}(\bm u_0) = 0.158279$. The control is discontinuous at the corners, see the last subfigure of \Cref{Figure_03}, and hence is not in $\bm H^{1/2}(\Gamma)$. Finite element error results are summarized in  \Cref{Table03}. Again we have $r\approx 1$ and $s\approx 0.5$. In this case, the observed experimental order of convergence for the $\bm L^2(\Gamma)$ error of the $\bm L^2(\Gamma)$-regularized problem is quite close to $s$.
	
	\begin{figure}
		\centering
		\includegraphics[width=.24\textwidth]{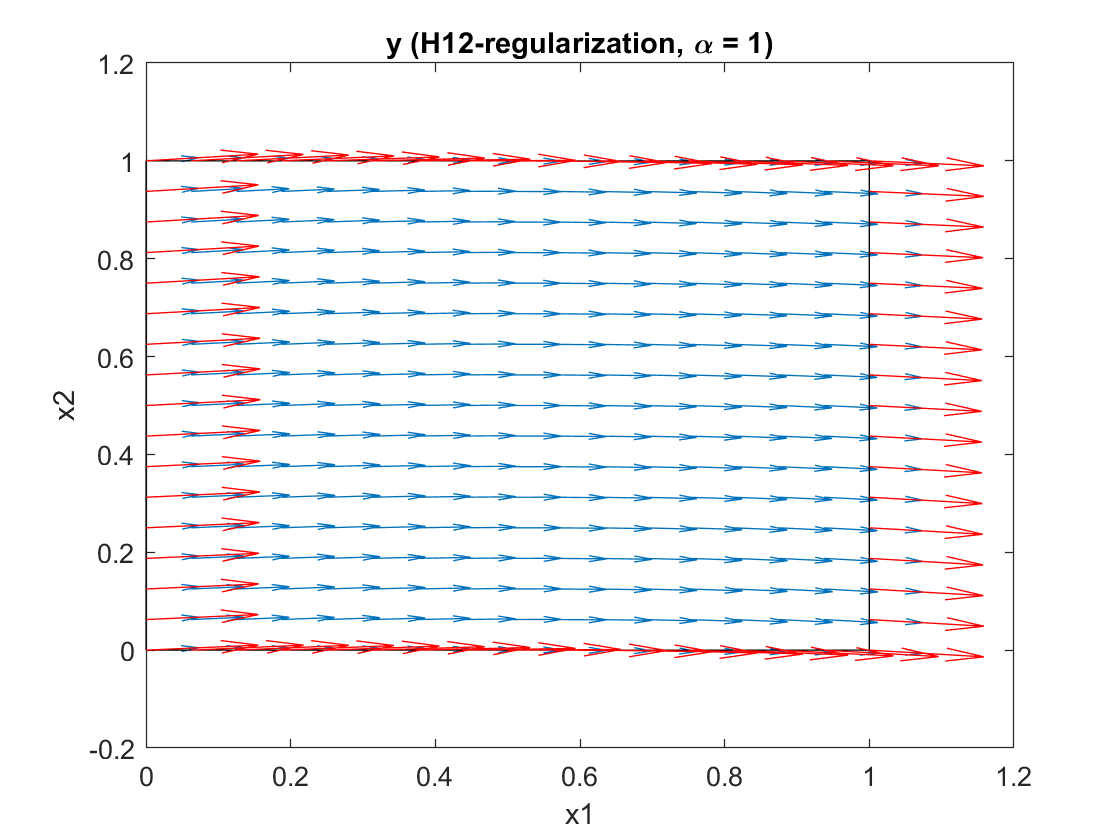}
		\includegraphics[width=.24\textwidth]{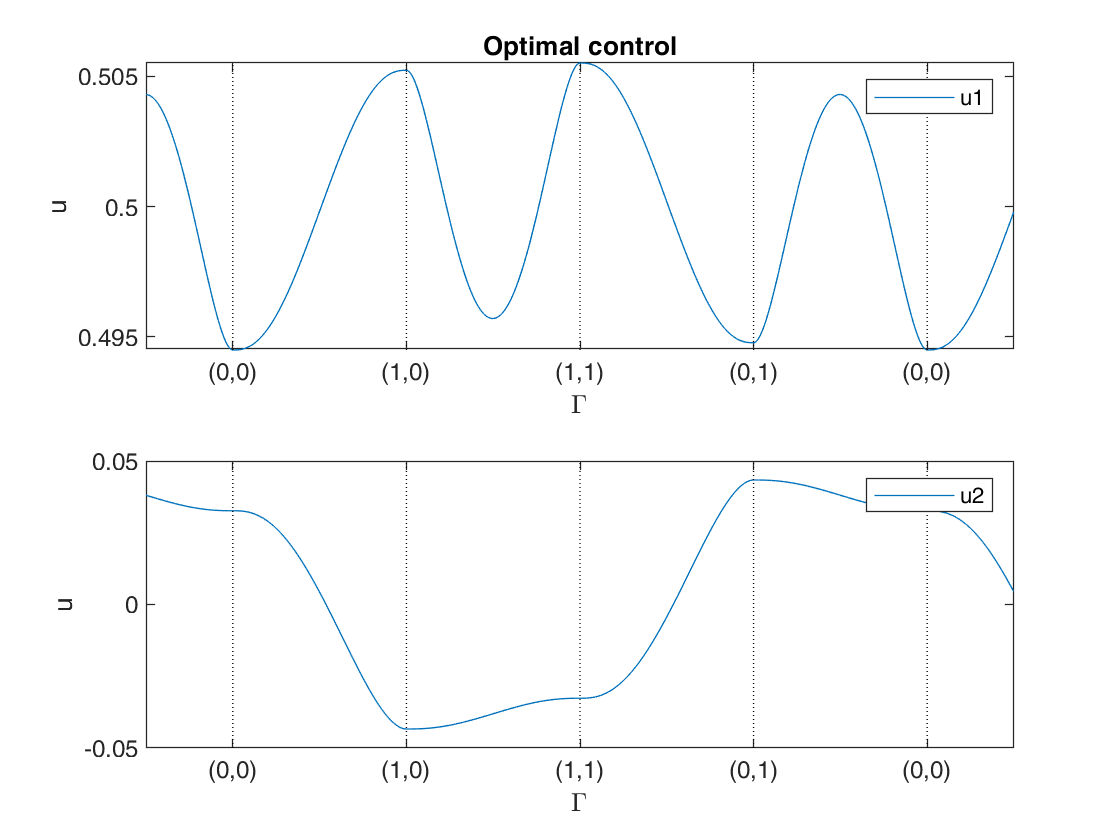}
		\includegraphics[width=.24\textwidth]{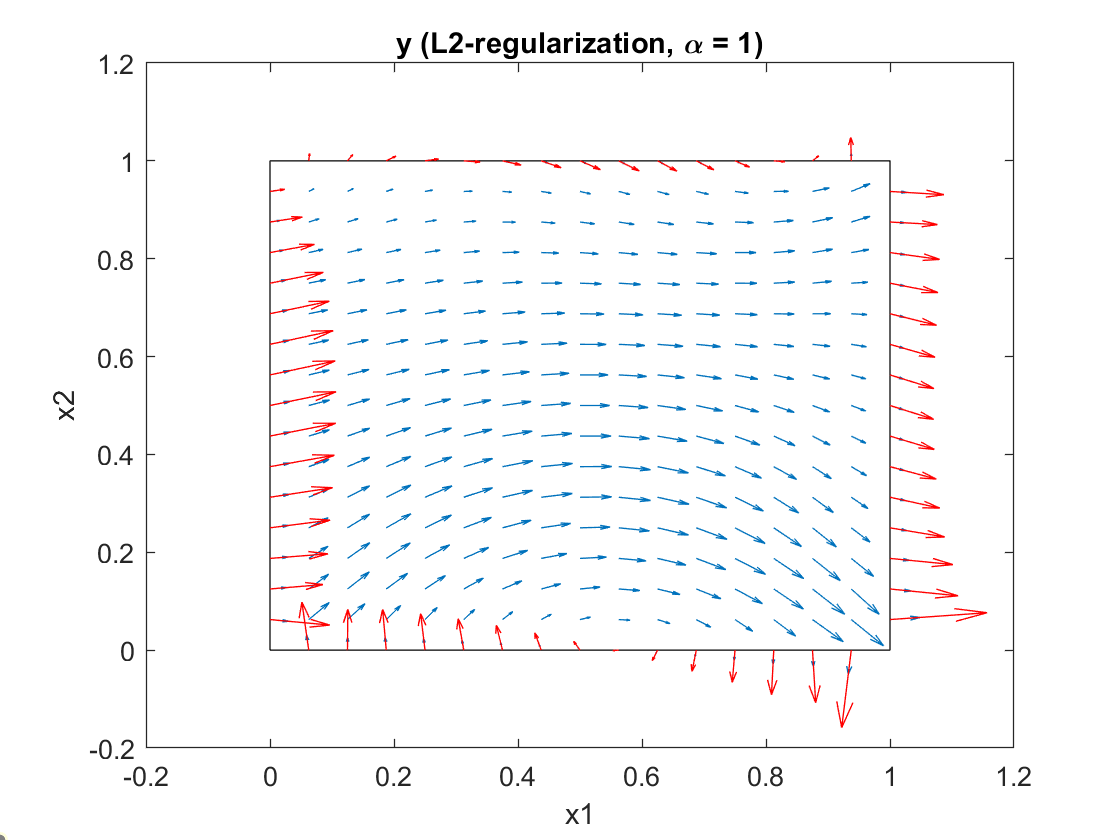}
		\includegraphics[width=.24\textwidth]{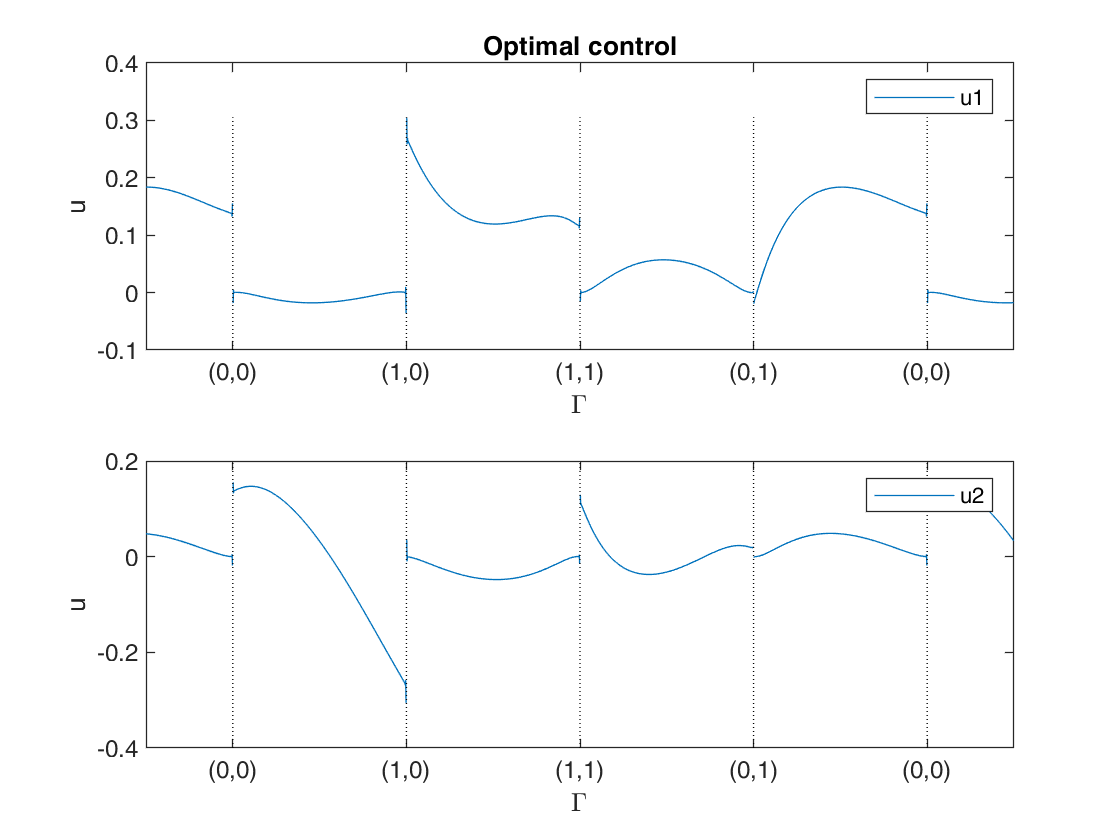}
	\caption{Solution of \Cref{Example2}:   The first two subfigures are  for  $\bm H^{1/2}(\Gamma)$ regularization, the last two subfigures are for  $\bm L^2(\Gamma)$ regularization.}\label{Figure_03}
	\end{figure}
	
	\begin{table}
	\centering
	\begin{tabular}{cc|c|c|c|c|c|c}
		\Xhline{1pt}

		\multirow{2}{*}{}
		&\multirow{2}{*}{$i$}	
		&\multicolumn{4}{c|}{$\bm H^{1/2}$ regularization}	
		&\multicolumn{2}{c}{$\bm L^{2}$ regularization}	\\
		\cline{3-8}
		& &$\|\bm E_h\|_{\bm H^{1/2}(\Gamma)}$ &Rate
		&$\|\bm E_h\|_{\bm L^2(\Gamma)}$ &Rate
		&$\|\bm E_h\|_{\bm L^2(\Gamma)}$ &Rate
		\\
		\cline{1-8}
		\multirow{5}{*}{}
		&	$2$	&	2.80E-02	&	-	    &	 3.77E-03	&	-	    &	 1.29E-01	&	-	   \\
		&	$3$	&	9.88E-03    &	1.50	&	  1.05E-03	&	1.85	&	 8.90E-02	&	0.53   \\
		&	$4$	&	 3.34E-03	&	1.57	&	  2.81E-04	&	1.90	&	 6.22E-02	&	0.52\\
		&	$5$	&	 1.10E-03	&	1.60	&	  7.32E-05	&	1.94	&	 4.37E-02	&	0.51\\
		&	$6$	&	3.67E-04	&	1.59	&	  1.86E-05	&	1.98	&	 3.08E-02	&	0.51\\

		\Xhline{1pt}

	\end{tabular}
	\caption{Errors and experimental order of convergence for  \Cref{Example2}.}\label{Table03}
\end{table}

\end{example}

\begin{example}\label{Example3}
	We take the same data as \Cref{Example2}, but now consider the L-shaped domain $\Omega=(-1,1)^2\setminus(0,1)^2$.
	The results on this domain are ${\bm F}(\bm 0) = 1.75$, ${\bm F}(\bar{\bm u}) = 1.107016$, ${\bm F}(\bm u_0) = 1.044080$. Graphs of the data and the solutions can be found in the right \Cref{Figure_01} and the first two subfigures of
	\Cref{Figure_04}. Experimental orders of convergence are in  \Cref{Table04}. The singular exponent for this domain is $\xi = 0.544$, so $r\approx 0.544$ and $s \approx 0.044$. The observed orders of convergence are higher.
	
	One remarkable fact is that for the $\bm L^2(\Gamma)$-regularized problem the optimal control need not tend to $\infty$ at a nonconvex corner, as happens with Dirichlet optimal control problems governed by the Poisson equation in a nonconvex polygonal domain.
	
	\begin{figure}
		\centering
		\includegraphics[width=.24\textwidth]{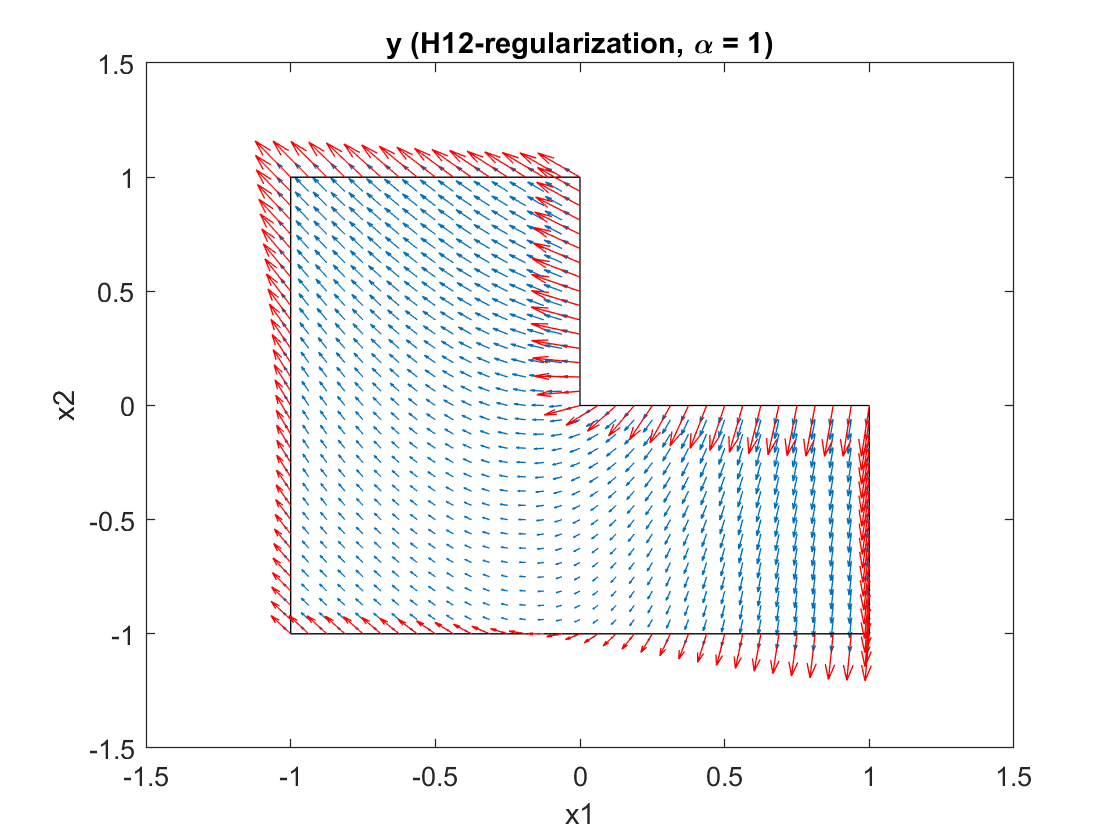}
		\includegraphics[width=.24\textwidth]{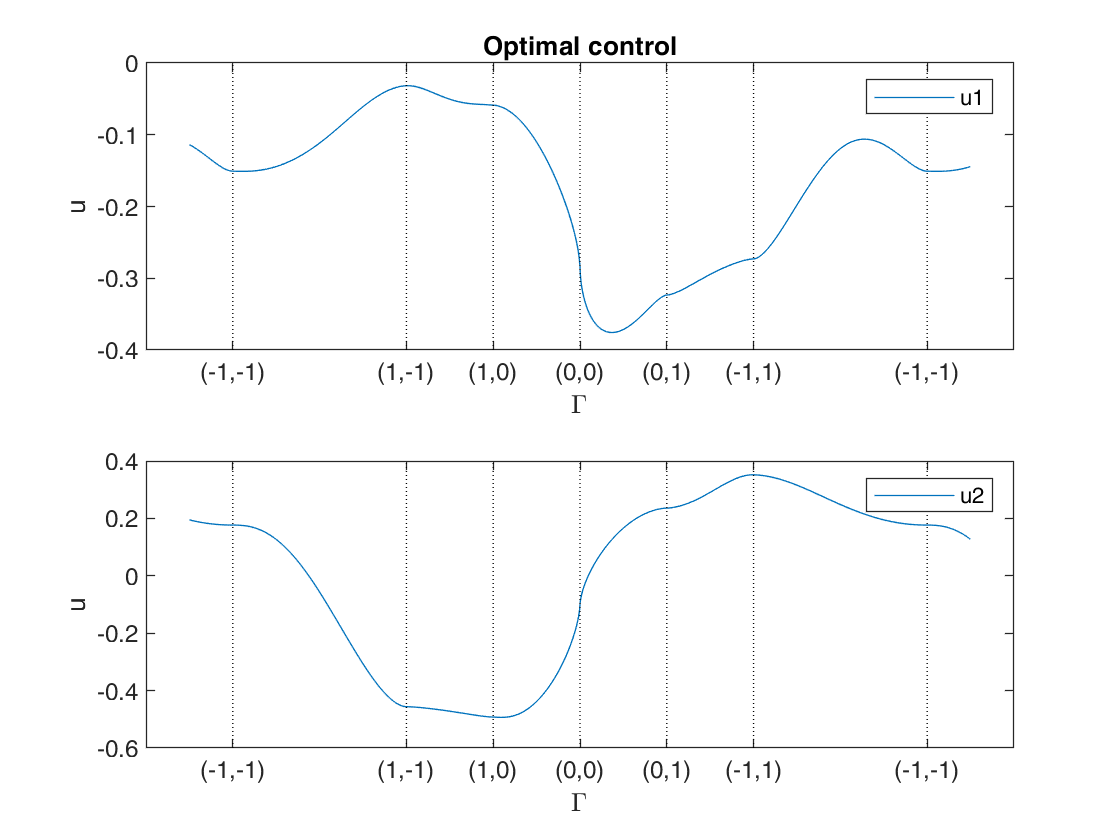}
		\includegraphics[width=.24\textwidth]{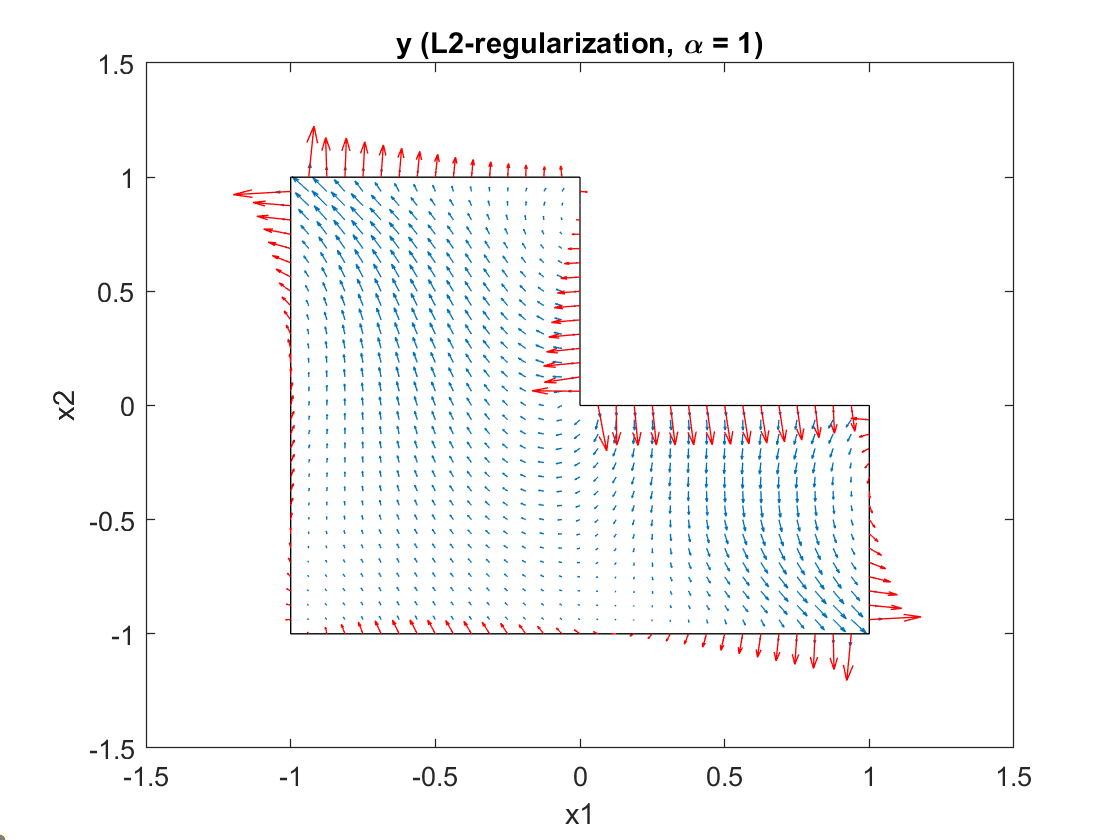}
		\includegraphics[width=.24\textwidth]{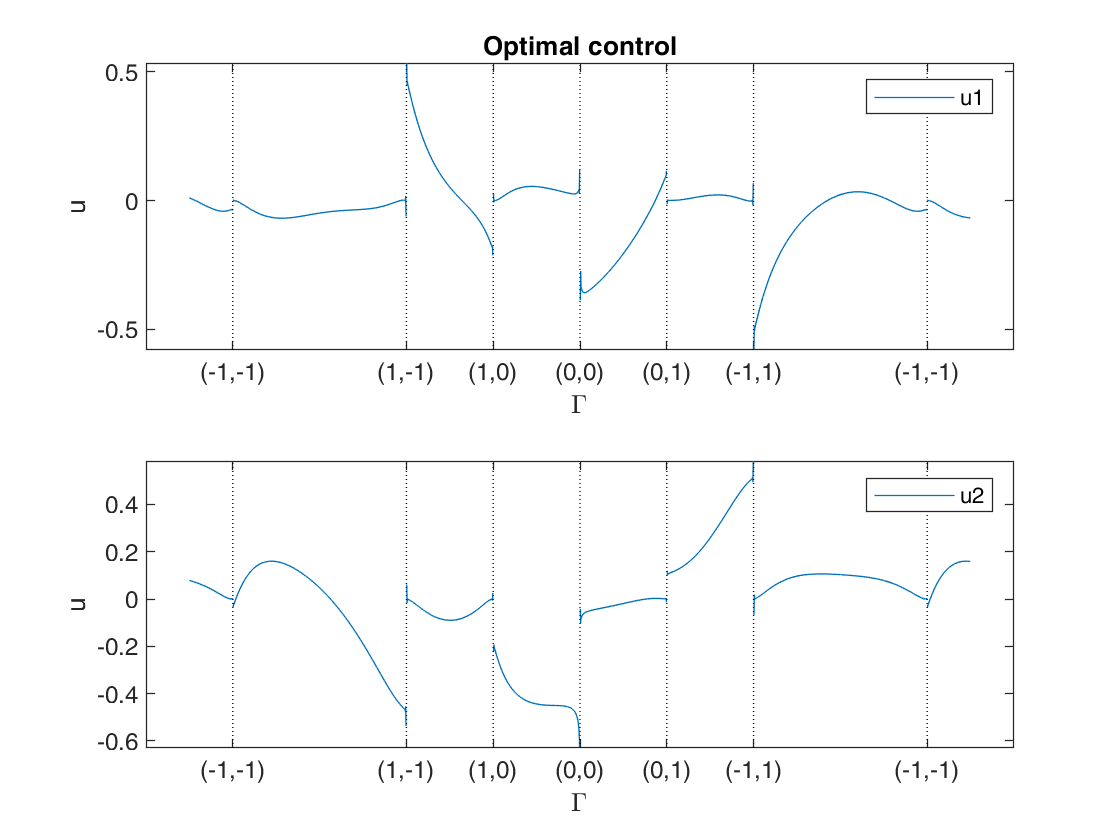}
		\caption{Solution of \Cref{Example3}:   The first two subfigures are  for  $\bm H^{1/2}(\Gamma)$ regularization, the last two subfigures are for  $\bm L^2(\Gamma)$ regularization.}\label{Figure_04}
	\end{figure}

		\begin{table}
		\centering
		\begin{tabular}{cc|c|c|c|c|c|c}
			\Xhline{1pt}

			\multirow{2}{*}{}
			&\multirow{2}{*}{$i$}	
			&\multicolumn{4}{c|}{$\bm H^{1/2}$ regularization}	
			&\multicolumn{2}{c}{$\bm L^{2}$ regularization}	\\
			\cline{3-8}
			& &$\|\bm E_h\|_{\bm H^{1/2}(\Gamma)}$ &Rate
			&$\|\bm E_h\|_{\bm L^2(\Gamma)}$ &Rate
			&$\|\bm E_h\|_{\bm L^2(\Gamma)}$ &Rate
			\\
			\cline{1-8}
			\multirow{5}{*}{}
			&	$2$	&	4.11E-01	&	-	    &	 7.40E-02	&	-	    &	 3.40E-01	&	-	   \\
			&	$3$	&	2.49E-01    &	0.72	&	  3.42E-02	&	1.12	&	 2.38E-01	&	0.51   \\
			&	$4$	&	 1.53E-01	&	0.71	&	  1.55E-02	&	1.14	&	 1.71E-01	&	0.48\\
			&	$5$	&	 9.12E-02	&	0.74	&	  6.86E-03	&	1.18	&	 1.24E-01	&	0.46\\
			&	$6$	&	5.07E-02	&	0.85	&	  2.83E-03	&	1.28	&	8.95E-02	&	0.47\\

			\Xhline{1pt}

		\end{tabular}
		\caption{Errors and experimental order of convergence for  \Cref{Example3}.}\label{Table04}
	\end{table}

\end{example}

\bibliographystyle{siamplain}
%
\bibliography{Dirichlet_Boundary_Control,Added,Mypapers}
\end{document}